\documentclass{article}

\usepackage{amsmath}
\usepackage{amssymb}
\usepackage{latexsym}
\usepackage{color}

\newtheorem{assumption}{Assumption}
\def\qed{ \ \vrule width.2cm height.2cm depth0cm\smallskip}
\newenvironment{proof}{\noindent {\bf Proof.\/}}{$\qed$\vskip 0.1in}

\newcommand{\la}{\langle}
\newcommand{\ra}{\rangle}

\newcommand{\we}{\wedge}
\newcommand{\ol}{\overline}
\newcommand{\ul}{\underline}
\newcommand{\eps}{\varepsilon}

\newcommand{\ba}{\begin{array}}
\newcommand{\ea}{\end{array}}
\newcommand{\be}{\begin{equation}}
\newcommand{\ee}{\end{equation}}
\newcommand{\bea}{\begin{eqnarray}}
\newcommand{\eea}{\end{eqnarray}}
\newcommand{\beaa}{\begin{eqnarray*}}
\newcommand{\eeaa}{\end{eqnarray*}}

\def\dbE{\mathbb{E}}
\def\dbF{\mathbb{F}}

\def\dbH{\mathbb{H}}
\def\dbI{\mathbb{I}}

\def\dbL{\mathbb{L}}

\def\dbP{\mathbb{P}}
\def\dbR{\mathbb{R}}
\def\dbS{\mathbb{S}}
\def\dbT{\mathbb{T}}
\def\dbQ{\mathbb{Q}}

\def\a{\alpha}
\def\b{\beta}
\def\g{\gamma}
\def\d{\delta}
\def\e{\varepsilon}

\def\l{\lambda}

\def\si{\sigma}
\def\t{\tau}
\def\f{\varphi}

\def\o{\omega}

\def\Th{\Theta}
\def\L{\Lambda}

\def\O{\Omega}
\def\cA{{\cal A}}
\def\cB{{\cal B}}

\def\cE{{\cal E}}
\def\cF{{\cal F}}
\def\cG{{\cal G}}

\def\cJ{{\cal J}}

\def\cL{{\cal L}}

\def\cP{{\cal P}}

\def\cT{{\cal T}}

\def\cY{{\cal Y}}
\def\cZ{{\cal Z}}

\def\ch{\textsc{h}}

\def\no{\noindent}

\def\ms{\medskip}
\def\bs{\bigskip}
\def\q{\quad}
\def\qq{\qquad}

\def\pa{\partial}
\def\cd{\cdot}
\def\cds{\cdots}

\def\qed{ \hfill \vrule width.25cm height.25cm depth0cm\smallskip}

\newcommand{\basa}{\begin{assumption}}
\newcommand{\easa}{\end{assumption}}

\newcommand{\bas}{\begin{assum}}
\newcommand{\eas}{\end{assum}}

\def\limsup{\mathop{\overline{\rm lim}}}
\def\liminf{\mathop{\underline{\rm lim}}}

\def\esup{\mathop{\rm ess\!-\!sup}}
\def\einf{\mathop{\rm ess\!-\!inf}}

\def\pa{\partial}

 \def\cd{\cdot}
\def\cds{\cdots}

\def\dis{\displaystyle}

\def\1{{\bf 1}}

\def\:{\!:\!}
\def\reff#1{{\rm(\ref{#1})}}
\def \proof{{\noindent \bf Proof\quad}}

at 9pt

\begin{document}

\newtheorem{thm}{Theorem}[section]
\newtheorem{lem}[thm]{Lemma}
\newtheorem{cor}[thm]{Corollary}
\newtheorem{prop}[thm]{Proposition}
\newtheorem{rem}[thm]{Remark}
\newtheorem{eg}[thm]{Example}
\newtheorem{defn}[thm]{Definition}
\newtheorem{assum}[thm]{Assumption}

\renewcommand {\theequation}{\arabic{section}.\arabic{equation}}
\def\thesection{\arabic{section}}

\numberwithin{equation}{section}
\numberwithin{thm}{section}

\title{\bf Comparison of Viscosity Solutions of Semi-linear Path-Dependent PDEs}
\author{Zhenjie  {\sc Ren}\footnote{CMAP, Ecole Polytechnique Paris, ren@cmap.polytechnique.fr. Research supported by grants from R\'egion Ile-de-France}   \and
Nizar {\sc Touzi}\footnote{CMAP, Ecole Polytechnique Paris, nizar.touzi@polytechnique.edu. Research supported by ANR the Chair {\it Financial Risks} of the {\it Risk Foundation} sponsored by Soci\'et\'e G\'en\'erale, and the Chair {\it Finance and Sustainable Development} sponsored by EDF and Calyon. }
       \and Jianfeng {\sc Zhang}\footnote{University of Southern California, Department of Mathematics, jianfenz@usc.edu. Research supported in part by NSF grant DMS 1413717.}}
\maketitle

\begin{abstract}
This paper provides a probabilistic proof of the comparison result for viscosity solutions of path-dependent semilinear PDEs. We consider the notion of viscosity solutions introduced in \cite{EKTZ} which considers as test functions all those smooth processes which are tangent in mean. When restricted to the Markovian case, this definition induces a larger set of test functions, and reduces to the notion of stochastic viscosity solutions analyzed in \cite{BayraktarSirbu1,BayraktarSirbu2}. Our main result takes advantage of this enlargement of the test functions, and provides an easier proof of comparison. This is most remarkable in the context of the linear path-dependent heat equation. As a key ingredient for our methodology, we introduce a notion of punctual differentiation, similar to the corresponding concept in the standard viscosity solutions \cite{CaffarelliCabre}, and we prove that semimartingales are almost everywhere punctually differentiable. This smoothness result can be viewed as the counterpart of the Aleksandroff smoothness result for convex functions. A similar comparison result was established earlier in \cite{EKTZ}. The result of this paper is more general and, more importantly, the arguments that we develop do not rely on any representation of the solution. 
\end{abstract}

\vspace{5mm}


\vspace{5mm}


\vfill\eject

\section{Introduction}

This paper provides a purely probabilistic wellposedness result for the semilinear path-dependent partial differential equation:
 \bea\label{ppde-intro}
 -\partial_tu-\frac12\mbox{Tr}\big[\si_t(\o)\si^{\rm T}_t(\o)\partial^2_{\o\o}u\big]
 -F_t\big(\o,u_t(\o), \si^T_t(\o)\pa_\o u_t(\o)\big)
 = 0
 &\mbox{on}&
 [0,T)\time\O,
 \eea
 where $T>0$ is a given terminal time, $\o\in \O$ is a continuous path from $[0,T]$ to $\dbR^d$ starting from the origin, the diffusion coefficient $\si$ is a mapping from $[0, T]\times \O$ to $\dbR^{d\times d}$ with $\si^T$ denoting its transpose, and the nonlinearity $F$ is a mapping from $(t,\o,y,z)\in[0,T]\times\O\times\dbR\times\dbR^d$ to $\dbR$.
The unknown process $\{u_t(\o),t\in[0,T]\}$ is required to be continuous in $(t,\o)$. In the smooth case, the derivatives $\partial_tu$, $\partial_{\o}u$, and $\partial^2_{\o\o}u$ are defined in accordance with the functional It\^{o} formula  introduced by Dupire \cite{Dupire}. However, as it is shown in the previous literature \cite{EKTZ,ETZ1,ETZ2}, such a smoothness requirement is rather exceptional, even in the case of the path-dependent heat equation, that is, $\si = I_d$ and $F\equiv 0$.

Our objective is to continue the development of the theory of viscosity solutions in this context. Viscosity solutions in finite dimensional spaces, which are locally compact, were introduced by Crandall and Lions \cite{CrandallLions}, we refer to \cite{CrandallIshiiLions} and \cite{FlemingSoner} for an overview. Extensions to infinite-dimensional spaces with special structure have also been  established by Lions \cite{Lions1,Lions2, Lions3} and Swiech \cite{Swiech}. However, these extensions are not suitable for our purpose due to the two following reasons. First, the path space $\O$ is a Banach space when endowed with the $\dbL^\infty-$norm, and not a Hilbert space as assumed in the above literature. Secondly, the adaptedness requirement on the functions $u(t,\o)$ is a special feature of our problem which is not addressed in the infinite-dimensional PDE literature. 

Nonlinear path-dependent PDEs appear in various applications as the stochastic control of non-Markovian systems \cite{ETZ1}, and the corresponding stochastic differential games \cite{PhamZhang}. They are also intimately related to the backward stochastic differential equations introduced by Pardoux and Peng \cite{PardouxPeng}, and their extension to the second order in \cite{CSTV,STZ2}. Loosely speaking, backward SDEs can be viewed as Sobolev solutions of path-dependent PDEs, and our goal is to develop the alternative notion of viscosity solutions which is well-known to provide a suitable wellposedness and stability theory in the Markovian case $u(t,\o)=u(t,\o_t)$. We also refer to the recent applications in \cite{H-LTT} to establish a representation of the solution of a class of equations  \eqref{ppde-intro} in terms of branching diffusions, and to \cite{MRTZ} for the small time large deviation results of path-dependent diffusions.

The notion of viscosity solutions studied in this paper, as introduced in \cite{EKTZ, ETZ1, ETZ2}, consider smooth test processes which are tangent in mean, with respect to an appropriate class of probability measures, to the process of interest. This is in contrast with the Crandall and Lions \cite{CrandallLions} standard notion of viscosity solutions in finite dimensional spaces where the test functions are tangent in the pointwise sense. In particular, when restricted to the Markovian case, our notion of viscosity solutions allows for a larger set of test functions, and in the case of the heat equation (or more general linear equation) case it reduces to the notion of stochastic viscosity solutions analyzed by Bayraktar and Sirbu \cite{BayraktarSirbu1,BayraktarSirbu2}. Consequently, the uniqueness may be easier with our notion, while existence is more restricted and may become harder. However, it was proved in the previous papers \cite{EKTZ,ETZ1,ETZ2} that existence still holds true under this notion of viscosity solutions for a large class of equations. In particular, in the present semilinear case, the solution of  backward SDEs provides a natural probabilistic representation for viscosity solution of path dependent PDEs, and thus extends the nonlinear Feynman-Kac formula of \cite{PP92} to path dependent case.


The main contribution of this paper is to provide a probabilistic proof of the comparison result for the path-dependent equation \eqref{ppde-intro} which, in contrast with \cite{EKTZ}, does not rely on any representation of the solution. We also observe that the present comparison result is stronger than that of \cite{EKTZ} as it holds in a larger class of processes and for a random and possibly degenerate diffusion coefficient $\si$ (in \cite{EKTZ}, only $\si=I_d$ is considered). Our proof by-passes completely the delicate and deep Crandall and Ishii Lemma (see Lemma 3.2 in \cite{CrandallIshiiLions}). In particular, our proof of comparison result for the path-dependent heat equation is elementary, and does not require any penalization to address (the standard comparison result for second order PDEs applies to a bounded domain, the extension to an unbounded domain involves a penalization using the growth conditions). In particular, the wellposedness of the path-dependent heat equation is a direct consequence of the equivalence between the viscosity subsolution and the  submartingale properties.

Our arguments are inspired from the work of Caffarelli and Cabre \cite{CaffarelliCabre}. By adapting the notion of punctual differentiation to our path-dependent framework, we prove an important smoothness result. Namely, we show that semimartingales are punctually differentiable Leb$\otimes\dbP-$a.e. This result can be viewed as the analogue of the Aleksandroff regularity result for convex functions. In the present semilinear case, an important property of our notion of viscosity solutions is that viscosity subsolutions (resp. supersolutions) are submartingales (resp. supermartingales) up to the addition of some absolutely continuous process. In particular, viscosity subsolutions and supersolutions are  punctually differentiable Leb$\otimes\dbP-$a.e.

We shall remark that, while the framework of fully nonlinear path dependent PDEs in \cite{ETZ1, ETZ2} covers the random coefficient $\si$ here, their comparison result excludes this case due to their heavy reliance  on the locally uniform smooth approximation of the viscosity solution. The definition of viscosity solutions here is slightly different from that  in \cite{ETZ1, ETZ2} by considering even more test functions. This enlargement of test function class allows us to establish the punctual differentiation of viscosity solutions, which does not require the smooth approximation to be locally uniform.  On a different perspective, the class of probability measures used to determine the test functions is non-dominated in \cite{ETZ1, ETZ2}, consequently the corresponding convergence theorem requires very strong regularity of the involved processes. It is still unclear how to obtain the punctual differentiation of viscosity solutions, even for the present semilinear PPDE \reff{ppde-intro}, if we use the non-dominated class of probability measures as in \cite{ETZ1, ETZ2}. 

The rest of the paper is organized as follows. Section \ref{sect:preliminary} introduces the set up of the problem, in particular the class of probability measures we will use. The notion of viscosity solution is defined in Section \ref{sec:semi PPDE}. In particular, similar to the Crandall and Lions \cite{CrandallLions} standard notion of viscosity solutions,  we show that our notion  for path dependent PDEs can be formulated equivalently in terms of the corresponding semijets. Section \ref{sec:present_mainresults} is devoted to the main result of the paper: the comparison result of viscosity solutions of semilinear path dependent PDEs. Then, Section \ref{sec:existence} prove briefly the existence of viscosity solutions by using the wellposedness of corresponding BSDEs. Finally Section \ref{sec:appendix} completes the technical proofs.

\section{Preliminaries}\label{sect:preliminary}
Throughout this paper let $T>0$ be a given finite maturity, $\O:=\{\o\in C([0,T];\dbR^d):\o_0=0\}$ the set of continuous paths starting from the origin, and $\Theta:=[0,T]\times\O$. We denote by $B$ the canonical process on $\O$, $\dbF = \{\cF_t, 0\le t\le T\}$ the canonical filtration,  $\cT$ the set of all $\dbF$-stopping times taking values in $[0,T]$, and $\dbP_0$ the Wiener measure on $\O$. Moreover, let $\cT^+$ denote the subset of $\t\in\cT$  taking values in $(0,T]$, and for $\ch \in \cT$, let $\cT_\ch$ and $\cT_\ch^+$ be the subset of $\t\in \cT$ taking values in $[0, \ch]$ and $(0, \ch]$, respectively.

Following Dupire \cite{Dupire}, we introduce the following pseudo-distance on $\Theta$:
 \beaa
\|\o\|_t := \sup_{0\le s\le t} |\o_s|,\q  d\big((t,\o),(t',\o')\big)
 :=
 |t-t'|+\|\o_{t\wedge}-\o'_{t'\wedge}\|_T
 &\mbox{for all}&
 (t,\o), (t',\o')\in\Theta.
 \eeaa
We say  a process valued in some metric space $E$ is in $C^0(\Theta,E)$ whenever it is continuous with respect to $d$. Similarly, $\dbL^0(\cF_t, E)$ and $\dbL^0(\dbF, E)$ denote the set of $\cF_t$-measurable random variables and $\dbF$-progressively measurable processes, respectively. We remark that $C^0(\Theta,E)\subset \dbL^0(\dbF, E)$, and when $E=\dbR$, we shall omit it in these notations.

For any $A\in \cF_T$, $\xi\in \dbL^0(\cF_T, E)$, $X\in \dbL^0(\dbF,E)$, and  $(t,\o)\in[0,T]\times\O$, define:
 \beaa
&A^{t,\o} := \{\o'\in \O: \o\otimes_t \o' \in A\},\q  \xi^{t,\o}(\o')
 :=
 \xi(\o\otimes_t\o'), 
\q  X^{t,\o}_s(\o')
 :=
 X(t+s,\o\otimes_t\o')&\\
 &~~\mbox{for all}~~\o'\in\O,~~
 \mbox{where}\q
 (\o\otimes_t\o')_s:=\o_s\1_{[0, t]}(s) +(\o_t+\o'_{s-t})\1_{(t, T]} (s),\q 0\le s\le T.&
 \eeaa
 Following the standard arguments of monotone class, we have the following simple results.

\begin{lem}\label{lem:measurability F}
Let $0\le t\le s\le T$ and $\o\in \O$. Then
 $A^{t,\o}\in \cF_{s-t}$ for all $A\in\cF_{s}$, $\xi^{t,\o} \in \dbL^0(\cF_{s-t}, E)$ for all $\xi\in \dbL^0(\cF_s, E)$, $X^{t,\o} \in \dbL^0(\dbF, E)$ for all  $X\in \dbL^0(\dbF, E)$, and $\t^{t,\o}-t\in \cT_{s-t}$ for all $\t\in\cT_s$.
\end{lem}

To study the semilinear PPDE \reff{ppde-intro}, we need to introduce the diffusion coefficient $\si$. Throughout the paper, the following assumption will always be in force.

\begin{assum}\label{assum:si}
The diffusion coefficient $\si: (t,\o)\in\Th\rightarrow \si_t(\o)\in \dbR^{d\times d}$ is  continuous in $t$, and Lipschtiz continuous in $\o$ uniformly in $t$, i.e.
$$|\si_t(\o)-\si_t(\o')|\leq C\|\o-\o'\|_t~~\mbox{for all}~~t\in[0,T],~\o,\o'\in\O,~~\mbox{for some}~~C\geq 0;$$
\end{assum}

\begin{rem}
Assumption \ref{assum:si} implies that $\si$ is continuous in $(t,\o)$, and thus $\dbF-$adapted. Also, we allow the parabolic  PPDE \reff{ppde-intro} to be degenerate.
\end{rem}

Our paper  builds on the following result.

\begin{lem}\label{lem:weaksolution}
For any bounded $\l \in \dbL^0(\dbF, \dbR^d)$, the following SDE has a unique weak solution:
\bea
\label{Xweak}
d X_t = \si(t, X_\cd) \big[dW_t + \l_t(X_\cd) dt\big],\q  X_0=0,
\eea
where $W$ is a Brownian motion. 
In particular, if $\l=0$, the SDE has a unique strong solution. The solution will be denoted as $\dbP_{\si,\l}$, and $\dbP_\si := \dbP_{\si, 0}$ when $\l=0$.
\end{lem}
\proof We first construct the solution by using the Girsanov transformation. First, thanks to the Lipschitz continuity of $\si$, let $X$ be the unique (strong) solution of the following SDE \bea
\label{Xsi}
X^\si_t = \int_0^t \si_s(X^\si_\cd)dB_s~~\mbox{for all}~~t,~\dbP_0-\mbox{a.s.}
\eea
Denote
\bea
\label{Psil}
&\dis B^\l_t := B_t - \int_0^t \l_t (X^\si_\cd)dt,\q  \dbP_{\si,\l} := \dbP_\l \circ (X^\si)^{-1}&\\
&\dis \mbox{where}~ {d \dbP_\l\over d\dbP_0}
:= 
\exp\Big(\int_0^T \l_t(X^\si_\cd) \cd dB_t - {1\over 2} \int_0^T |\l_t(X^\si_\cd)|^2dt\Big).&\nonumber
\eea
Then clearly $(B^\l, X^\si, \dbP_{\si,\l})$ is a weak solution to SDE \reff{Xweak}.

For the uniqueness, we follows the arguments in \cite{KS} Proposition 5.3.10.  Let $(W^i, X^i, \dbP^i)$, $i=1,2$ be two weak solutions to SDE \reff{Xweak}, namely $W^i$ is a $\dbP^i$-Brownian motion and
\beaa
d X^i_t = \si(t, X^i_\cd) \big[dW^i_t + \l_t(X^i_\cd) dt\big],\q  X^i_0=0,\q \dbP^i\mbox{-a.s.}
\eeaa
Denote 
\beaa
\tilde W^i_t := \int_0^t [dW^i_s + \l_s(X^i_\cd) ds],\q {d\tilde \dbP^i\over d\dbP^i} := M^i_T := \exp\Big(-\int_0^T \l_t(X^i_\cd) \cd dW_t^i - {1\over 2}\int_0^T |\l_t(X^i_\cd)|^2dt\Big).
\eeaa
Then $\tilde W^i$ is a $\tilde\dbP^i$-Brownian motion, and 
\beaa
X^i_t = \int_0^t \si( s,X^i_\cd) d\tilde W^i_s, \q \tilde\dbP^i\mbox{-a.s.}
\eeaa
By the Lipschitz continuity of $\si$, the $\tilde \dbP^1$-distribution of $(\tilde W^1, X^1)$ is equal to the  $\tilde \dbP^2$-distribution of $(\tilde W^2, X^2)$.  Note that $W^i$ and $M^i_T$ are functions of $(\tilde W^i, X^i)$, we see that the $\tilde \dbP^1$-distribution of $(\tilde W^1, X^1, W^1, M^1_T)$ is equal to the  $\tilde \dbP^2$-distribution of $(\tilde W^2, X^2, W^2, M^2_T)$. Now it follows from  $d\dbP^i = (M^i_T)^{-1} d\tilde \dbP^i$ that  the $\dbP^1$-distribution of $X^1$ is equal to the  $\dbP^2$-distribution of $X^2$.
\qed

For any $\t\in \cT$ and $\o \in \O$, let $\dbP^{\t,\o}$ be an r.c.p.d. of the probability measure $\dbP$ conditional to $\cF_\t$.  

\begin{lem}\label{lemma full}
Let $M$ be a $\dbP$-martingale with continuous paths, $\dbP$-a.s.. Then, for any $\t\in\cT$ we have 
 \be\label{P0 full}
 \dbP\big[\O_\t^{0}\big]=1,
 ~\mbox{where}~
 \O_\t^{0}:=\big\{\o: M^{\t,\o}~\mbox{is a}~\dbP^{\t,\o}-\mbox{martingale}\big\}.
 \ee
\end{lem}

\begin{proof} By standard approximation arguments, 
it is sufficient to prove that, for any $0\le t_1 <t_2\le T$ and  any sequence $0\le s_1\le \cdots\le s_n\le t_1$ such that $(s_1,\cdots,s_n)\in \dbQ^n$, it holds that
\begin{multline}
\label{martingaleproof}
\dbE^{\dbP^{\t,\o}}\big[(M^{\t,\o}_{t_2}-M^{\t,\o}_{t_1})\f(B_{s_1},\cdots,B_{s_n})]=0,~\mbox{for all}~\f\in C_b(\dbR^n)\\
\mbox{and for all}~\o\in\O^n_\t,~\mbox{for some}~\O^n_\t~\mbox{such that}~\dbP[\O^n_\t]=1.
\end{multline}
Since $C_b(\dbR^n)$ is a separable space, there exists a countable set $(\psi^n_k)_{k\ge 1}$ dense in $C_b(\dbR^n)$. By the tower property, we may find $\O^n_k\subset\O$ such that
 \beaa
 \dbE^{\dbP^{\t,\o}}\big[(M^{\t,\o}_{t_2}-M^{\t,\o}_{t_1})\psi^n_k(B_{s_1},\cdots,B_{s_n})\big]=0, 
 &\mbox{for all}&
 \o\in\O^n_k,~~\dbP[\O^n_k]=1.
\eeaa
Then \eqref{martingaleproof} holds on $\O^n_\t:=\cap_{k\ge 1}\O^n_k$.
\end{proof}

For all $\t\in \cT$ and $\o \in \O$, it is clear that $\si^{\t,\o}$ satisfies Assumption \ref{assum:si}. Then, for any bounded $\l\in\dbL^0(\dbF,\dbR^d)$, we may define from Lemma \ref{lem:weaksolution} a probability measure $\dbP_{\si^{\t,\o}, \l^{\t,\o}}$. The next result compares this probability measure to the r.c.p.d. $\dbP^{\tau,\o}_{\si,\l}$.

\begin{prop}\label{prop:rcpd}
Let $\l \in \dbL^0(\dbF,\dbR^d)$ be  bounded and $\t\in \cT$. Then for $\dbP_{\si,\l}$-a.e. $\o$, $\dbP^{\t,\o}_{\si, \l} = \dbP_{\si^{\t,\o}, \l^{\t,\o}}$, namely $\dbE^{\dbP^{\t,\o}_{\si,\l}}[\xi] = \dbE^{\dbP_{\si^{\t,\o}, \l^{\t,\o}}}[\xi]$, for any bounded  $\xi \in \dbL^0(\cF_{T-\t(\o)})$.
\end{prop}
\proof  First, denote 
\beaa
M_t := B_t- \int_0^t (\si\l)_s (B_\cd)ds,\q N_t := M_tM_t^T - \int_0^t (\si\si^T)_s(B_\cd) ds.
\eeaa
By the uniqueness of weak solution of SDE \reff{Xweak} we see that a probability measure $\dbP$ is equal to $\dbP_{\si,\l}$ if and only if 
 both $M$ and $N$ are $\dbP$-martingales. Note that $M$ and $N$ are continuous. By Lemma \ref{lemma full}, for $\dbP_{\si,\l}$-a.e. $\o$,  it holds that
\beaa
M^{\t,\o}_t = B_t- \int_0^t (\si\l)^{\t,\o}_s (B_\cd)ds,\q N^{\t,\o}_t = (M_tM_t^T)^{\t,\o}_t - \int_0^t (\si\si^T)^{\t,\o}_s(B_\cd) ds
\eeaa
are $\dbP^{\t,\o}_{\si,\l}$-martingales, which implies that $\dbP^{\t,\o}_{\si,\l} = \dbP_{\si^{\t,\o}, \l^{\t,\o}}$.
\qed

We now introduce an important family of probability measures on $\O$: for $L\ge 0$ and $(t,\o) \in \Th$,
 \bea
 \label{cP}
 \cP_L^{t,\o}
 :=
 \Big\{\dbP_{\si^{t,\o},\l}: \l\in \dbL_L(\dbF)\Big\}~~\mbox{where}~~\dbL_L(\dbF):=\{\l\in \dbL^0(\dbF): |\l|\le L\},
 \quad\mbox{and}\quad
 \cP_L:=\cP_L^{0,0},
 \eea
 and the associated nonlinear expectations
 \bea
 \label{cE}
 \ol\cE_L^{t,\o}:=\sup_{\dbP\in\cP_L^{t,\o}}\dbE^{\dbP},\q
 \ul\cE_L^{t,\o}:=\inf_{\dbP\in\cP_L^{t,\o}}\dbE^{\dbP}, 
 &\mbox{and}& 
 \ol\cE_L := \ol \cE_L^{0,0},\q \ul \cE_L := \ul \cE^{0,0}_L.
 \eea
 Unlike \cite{ETZ0, ETZ1, ETZ2} where mutually singular measures are considered for fully nonlinear PPDEs, here all measures $\dbP\in \cup_{L>0} \cP_L$ are equivalent to $\dbP_\si$. In particular, for $\l\in \dbL_L(\dbF)$ and $\xi\in \dbL^0(\cF_T)$, we have,  using the notations in \reff{Xsi} and \reff{Psil},
 \beaa
 \dbE^{\dbP_{\si,\l}}[|\xi|] =  \dbE^{\dbP_{\l}}[|\xi(X^\si)|] = \dbE^{\dbP_0}[M^\l_T|\xi(X^\si)|] \le C \Big(\dbE^{\dbP_0}[|\xi(X^\si)|^{1+\e}]\Big)^{1\over 1+\e} =  C \Big(\dbE^{\dbP_\si}[|\xi|^{1+\e}]\Big)^{1\over 1+\e},
 \eeaa
for some constant $C=C_{L,\e}$.  That is,
 \bea
 \label{cEest}
 \ol\cE_L[|\xi|] \le C_{L,\e} \Big(\dbE^{\dbP_\si}[|\xi|^{1+\e}]\Big)^{1\over 1+\e}.
 \eea
 A direct consequence of this is the following  convergence theorem, which makes some analysis in this paper much easier than that in \cite{ETZ0, ETZ1, ETZ2}.

\begin{prop}
\label{prop:DCT}
Assume $\xi_n, \xi \in \dbL^0(\cF_T)$, $\xi_n \to \xi$ in probability $\dbP_\si$, and $\sup_n \dbE^{\dbP_\si}[|\xi_n|^{1+\e}] <\infty$  for some $\e>0$. Then $\lim_{n\to\infty}\ol \cE_L[|\xi_n-\xi|] = 0$ for all $L>0$.
\end{prop}

\section{Viscosity solutions of semilinear path dependent PDEs}
\label{sec:semi PPDE}

The objective of this paper is the semilinear path dependent PDEs \reff{ppde-intro}, which we rewrite as:
 \bea\label{PPDE}
 &-\cL u_t(\o)-F\big(t,\o,u_t(\o), \si^T_t(\o)\pa_\o u_t(\o)\big)=0,\q (t,\o) \in [0, T)\times \O,&\\
 &\mbox{where}~\cL u_t(\o):=\pa_t u_t(\o)+\frac12\text{Tr}[\si_t(\o)\si^{\rm T}_t(\o)\pa_{\o\o}^2 u_t(\o)],&\nonumber
 \eea
and the nonlinearity $F: (t,\o,y,z)\in\Theta\times\dbR\times\dbR^d \to \dbR$ is $\dbF$-progressively measurable in all variables. We shall assume

\begin{assum}\label{assum:F}
{\rm (i)}\quad $F$ is uniformly $L_0-$Lipschitz continuous in $(y,z)$, for some $L_0\ge 0$, i.e.
 \beaa
 |F(\cd,y,z)-F(\cd,y',z')|
 \le
 L_0\left(|y-y'|+|z-z'|\right)
 &\mbox{for all}&
 y,y'\in\dbR,~z,z'\in\dbR^d,
 \eeaa
{\rm (ii)}\quad There exists $F^0\in C^0(\Th)$ such that $|F(\cd, 0,0)|\le F^0$. 
\end{assum}

\subsection{Definition via test functions}

In the present semilinear case, we shall use the following notion of smoothness of processes.

\begin{defn}[Smooth processes]\label{def:C12}
We say that $u\in C^{1,2}(\Theta)$, if $u\in C^0(\Theta,\dbR)$ and there exist processes $\L,Z$ in $C^0(\Theta,\dbR)$ and $C^0(\Theta,\dbR^d)$, respectively, such that: for each $(t,\o)\in\Th$,
 \beaa
 u^{t,\o}_s-u_t(\o)
 &=&
 \int_0^s \L^{t,\o}_rdr + \int_0^s Z^{t,\o}_r\cdot dB_r~~\mbox{for all}~~s\in[0,T-t],
 ~\dbP_{\si^{t,\o}}\mbox{-a.s.}
 \eeaa
We denote $\cL u_t(\o):=\L_t(\o)$, $\partial_\omega u_t(\o):=Z_t(\o)$.
\end{defn}

\begin{rem}
\label{rem:Ito}{\rm
(i) Notice that all measures in $\cup_{L>0} \cP^{t,\o}_L$ are equivalent to $\dbP_{\si}^{t,\o}$. Then for $u\in C^{1,2}(\Theta)$, by definition the following functional It\^{o} formula in the spirit of Dupire holds:
\bea
\label{Ito}
u^{t,\o}_s-u_t(\o)
 =
 \int_0^s (\cL u)^{t,\o}_r dr + \int_0^s (\pa_\o u)^{t,\o}_r\cdot dB_r, ~~s\in[0,T-t],
 ~\dbP\mbox{-a.s. for all} ~\dbP\in \cup_{L>0} \cP^{t,\o}_L.
 \eea

(ii) Unlike \cite{EKTZ, ETZ1,ETZ2} where $\pa_t u$ and $\pa^2_{\o\o} u$ are defined separately, here they appear jointly in the term $\pa_t u + {1\over 2} Tr(\si\si^T \pa^2_{\o\o} u)$, following Dupire's functinal It\^o formula. Since $\si$ is given, for our purpose we do not need to distinguish the two terms and thus identify $\cL u_t(\o)$ directly with $\pa_t u + {1\over 2} Tr(\si\si^T \pa^2_{\o\o} u)$.
\qed}
\end{rem}

We introduce the sets of test processes for subsolutions and supersolutions:
 \bea
 \label{cA}
 \left.\ba{lll}
 \underline{\cA}_L u_t(\o)
 :=
 \Big\{\varphi\in C^{1,2}(\Theta):
          (\varphi-u)_t(\o)
          =
          \min_{\tau\in\cT_\ch} \underline{\cE}_L^{t,\o}\big[(\varphi-u)^{t,\o}_{\tau}\big]
          ~\mbox{for some}~\ch\in\cT^+_{T-t}
 \Big\},
 \\
 \overline{\cA}_L u_t(\omega)
 :=
 \Big\{\varphi\in C^{1,2}(\Theta):
          (\varphi-u)_t(\o)
          =
          \max_{\tau\in\cT_\ch} \overline{\cE}_L^{t,\o}\big[(\varphi-u)^{t,\o}_{\tau}\big]
          ~\mbox{for some}~\ch\in\cT^+_{T-t}
 \Big\}.
 \ea\right.
 \eea
The stopping time $\ch$ implies that the test processes are locally defined at $(t,\o)$, and in particular the integrability in \reff{cA} will always be guaranteed. For a test function $\f\in\ul\cA_L u_t(\o)\cup \ol\cA_Lu_t(\o)$, we shall refer to a corresponding $\ch$ as its localizing time. Note that in our definition, a test function is tangent to $u$ at a point $(t,\o)$ in mean value (under a family of probability measures), which is different from the corresponding notion in Crandall and Lions \cite{CrandallLions}.

\begin{defn}[Viscosity solution of path-dependent PDE]\label{def:visco} Let $u\in C^0(\Theta,\dbR)$.
\\
{\rm (i)} $u$ is a $\cP_L$-viscosity subsolution of \eqref{PPDE} if for any $(t,\o)\in[0,T)\times\O$:
 \beaa
 -\cL\varphi_t(\o)
         -F\big(t,\o,u_t(\o), \si^T_t(\o)\partial_\omega\varphi_t(\o)\big)
 \le
 0
 &\mbox{for all}&
 \f\in\ul\cA_L u_t(\o).
 \eeaa
{\rm (ii)} $u$ is a $\cP_L$-viscosity supersolution of \eqref{PPDE} if for any $(t,\o)\in[0,T)\times\O$:
 \beaa
 -\cL\varphi_t(\o)
         -F\big(t,\o,u_t(\o), \si^T_t(\o)\partial_\omega\varphi_t(\o)\big)
 \ge
 0
 &\mbox{for all}&
 \f\in\ol\cA_L u_t(\o).
 \eeaa
{\rm (iii)} A $\cP_L$-viscosity solution of \eqref{PPDE} is both a $\cP_L$-subsolution and a $\cP_L$-supersolution.
\end{defn}

\begin{rem}
\label{rem:test}
{\rm (i) In \cite{ETZ1, ETZ2}, a larger and non-dominated set $\ol \cP_L$ which consists of mutually singular probability measures is used. The corresponding sets of test functions $\ul{\cA}^{\ol{\cP}_L}_L u$ and $\ol{\cA}^{\ol{\cP}_L}_L u$ are smaller there, and consequently a viscosity solution here must be a viscosity solution in the sense of \cite{ETZ1}, but not vice versa in general. Therefore, by putting more test functions in this paper, we are helping for the proof of uniqueness.

(ii) When $\si = I_d$ but under the above $\ol\cP_L$-definition, the wellposedness of the semilinear PPDE \reff{PPDE} is achieved in \cite{ETZ1, ETZ2} by using a different approach.  However, the general case with random $\si$ and under $\ol\cP_L$-definition does not fall into the framework of this paper  and does not satisfy the sufficient conditions for comparison principle in  \cite{ETZ2},  and its wellposedness is still open.
\qed}
\end{rem}

\subsection{Equivalent definition via semijets}

Following the standard theory of viscosity solutions for PDEs, we may also define viscosity solutions via semijets. In light of Definition \ref{def:C12} and Remark \ref{rem:Ito} (ii), we introduce the linear processes:
 \bea
 \label{Q}
 Q^{\a,\b}(t,\o)
 :=
 \a t+\b\cdot\omega_t,
 &\a\in\dbR,~\b\in\dbR^d,&
 \mbox{and}~(t,\o)\in\Theta.
 \eea

\begin{defn}[Semijets]\label{Def jet}
For $u\in C^0(\Theta,\dbR)$, the subjet and superjet of $u$ at $(t,\o)$ are defined as:
 \beaa
 \ul\cJ_L u_t(\o)
 &:=&
 \big\{(\a,\b)\in \dbR\times\dbR^d:~Q^{a,\b}\in\ul\cA_L u_t(\o)
 \big\};\\
   \ol\cJ_L u_t(\o)
 &:=&
 \big\{(\a,\b)\in \dbR\times\dbR^d:~Q^{a,\b}\in\ol\cA_L u_t(\o)
 \big\}.
 \eeaa
Moreover,  $cl(\ul\cJ_L u_t(\omega))$ and $ cl(\ol\cJ_L u_t(\omega))$ denote their closures.
 \end{defn}

\begin{rem}
\label{rem:semijet}
{\rm In the fully nonlinear case, one has to distinguish $\pa_t u$ and $\pa^2_{\o\o}u$, and accordingly one needs to introduce paraboloid processes:
\beaa
 Q^{\a,\b,\g}(t,\o)
 :=
 \a t+\b\cdot\omega_t + {1\over 2} \g \o_t \cd \o_t,
 &\a\in\dbR,~\b\in\dbR^d,~ \g \in \dbR^{d\times d}&
 \mbox{and}~(t,\o)\in\Theta.
 \eeaa
 See \cite{Survey} for more details. In the present semilinear case, one can easily show that the linear processes \reff{Q} is sufficient for our purpose.
 \qed}
 \end{rem}

\begin{prop}
\label{prop:equiv}
Let $u\in C^0(\Theta,\dbR)$. Then the following are equivalent: for any $(t,\o)\in [0, T)\times \O$,

\no {\rm (i)}\q $u$ is a $\cP_L$-viscosity subsolution of the path-dependent PDE \eqref{PPDE} at $(t,\o)$;

\no {\rm (ii)}\q $-\a-F\big(t,\omega,u_t(\omega),\si^T_t(\o)\b\big)\le 0$ for all $(\a,\b)\in\ul\cJ_L u_t(\omega)$;

\no {\rm (iii)}\q $-\a-F\big(t,\omega,u_t(\omega),\si^T_t(\o)\b\big)\le 0$ for all $(\a,\b)\in cl(\ul\cJ_L u_t(\omega))$. 
\end{prop}

\proof Since $Q^{\a,\b}\in C^{1,2}(\Th,\dbR)$, clearly (i) implies (ii). Now assume (ii) holds true. For any $(\a,\b)\in cl(\ul\cJ_L u_t(\omega))$, there exist $(\a_n,\b_n)\in\ul\cJ_L u_t(\omega)$ such that $(\a_n, \b_n)\to (\a,\b)$. By (ii) we have $-\a_n-F\big(t,\omega,u_t(\omega),\si^T_t(\o)\b_n\big)\le 0$. Sending $n\to \infty$ we prove (iii).

It remains to prove that (iii) implies (i).  Let $(t,\o)\in[0,T)\times\Omega$ and $\varphi\in\underline\cA_L u_t(\omega)$ with localizing time $\ch\in \cT^+_{T-t}$. Without loss of generality, we take $(t,\omega)=(0,0)$ and $(\varphi-u)_0=0$. Denote
 \bea
 \label{ab}
 \a:=\cL\f_0,
 &\b:=\partial_\omega\varphi_0.
 \eea
For any $\e>0$, since $\si\in C^0(\Th)$  and  $\f$ is smooth, by otherwise choosing a smaller $\ch$ we may assume 
 \beaa
|\si_t-\si_0|\le 1,\q |\cL\f_t-\a|\le\eps,
 &|\partial_\omega\varphi_t-\b|\le\eps,& 0\le t\le \ch.
 \eeaa
Denote $\a_\eps:=\a+[1+L(1+|\si_0|)]\e$. Then, for all  $\tau\in\cT_\ch$,
 \beaa
 &&(Q^{\a_\eps,\beta}-u)_0
 -\underline\cE_L\!\big[(Q^{\a_\eps,\beta}-u)_{\tau}\big] = \ol\cE_L\big[(u - u_0 - Q^{\a_\e,\b})_{\t}\big]\\
 &\le&\ol\cE_L\big[(u -\f)_{\t}\big]+
 \overline\cE_L\!\big[(\varphi\!-\!\varphi_0\!-\!Q^{\a_\eps,\beta})_{\tau}\big]
 \le 
 \overline\cE_L\Big[\int_0^{\t}
                                  (\cL\f_s -\! \a_\eps)ds
                                  +(\partial_\omega\varphi_s - \b) \cd dB_s
                          \Big].
 \eeaa
 where the last inequality thanks to the fact that $\varphi\in\underline\cA_L u_0 $. Note that, for any $\l\in \dbL_L(\dbF)$, 
 \beaa
&&  \dbE^{\dbP_{\si,\l}}\Big[\int_0^{\t}
                                  (\cL\f_s -\! \a_\eps)ds
                                  +\int_0^\t (\partial_\omega\varphi_s - \b) \cd dB_s
                          \Big] \\
 &=& \dbE^{\dbP_{\si,\l}}\Big[\int_0^{\t}
                                  (\cL\f_s - \a)ds
                                  +\int_0^\t (\partial_\omega\varphi_s - \b) \cd \si_s \l_sds - [1+L(1+|\si_0|)]\e\t
                          \Big]\\
                          &\le& \dbE^{\dbP_{\si,\l}}\Big[\int_0^{\t}[
                                 \e 
                                  +\e L(1+|\si_0|) ] ds - [1+L(1+|\si_0|)]\e\t
                          \Big] =0.
 \eeaa
 By the arbitrariness of $\l\in \dbL_L(\dbF)$, we see that 
 \beaa
 (Q^{\a_\eps,\beta}-u)_0
 -\underline\cE_L\!\big[(Q^{\a_\eps,\beta}-u)_{\tau}\big]  \le  \overline\cE_L\Big[\int_0^{\t}
                                  (\cL\f_s -\! \a_\eps)ds
                                  +(\partial_\omega\varphi_s - \b) \cd dB_s
                          \Big] \le 0.
                          \eeaa             
That is,  $(\a_\eps,\b)\in\underline\cJ_L u_0$ and thus $(\a,\b)\in cl(\underline\cJ_L u_0)$. Now it follows from (iii) that
\beaa
-\a-F(0,0,u_0,\si^T_0\b)\le 0,
\eeaa
 which, together with \reff{ab}, exactly means (i).
\qed

The following simple results will be useful later.
\begin{prop}\label{new sum derv} Let $u, u'\in C^0(\Th, \dbR)$ and $(t,\o) \in \Th$.

\no {\rm (i)}\q $(\a,\b)\in \ul \cJ_L u_t(\o)$ if and only if $(-\a, -\b) \in  \ol \cJ_L (-u)_t(\o)$.

\no {\rm (ii)}\q If $(\a,\b)\in \ul \cJ_L u_t(\o)$, $(\a',\b')\in \ul \cJ_L u'_t(\o)$, then $(\a+\a',\b+\b')\in  \ul \cJ_L (u+u')_t(\o)$.

\no Moreover, the results remain true if we replace the semi-jets with their closures. 
\end{prop}

\begin{proof} (i) is obvious, and we can easily extend the results from semi-jets to their closures. It remains to prove  (ii). Indeed, by definition, there exists a common $\ch\in \cT^+_{T-t}$ such that
\beaa
u_t(\o) \ge \ol\cE_L[(u^{t,\o}-Q^{\a,\b})_{\t}],\q u'_t(\o) \ge \ol\cE_L[((u')^{t,\o}-Q^{\a',\b'})_{\t}],\q \forall \t\in \cT_\ch.
\eeaa
Then, by the sub-linearity of $\ol\cE_L$ we have
\beaa
(u+u')_t(\o) \ge \ol\cE_L\Big[(u^{t,\o}-Q^{\a,\b})_{\t}+((u')^{t,\o}-Q^{\a',\b'})_{\t}\Big] =\ol\cE_L\Big[\big([u+u']^{t,\o}-Q^{\a+\a',\b+\b'})_{\t}\Big].
\eeaa
This means that $(\a+\a',\b+\b')\in  \ul \cJ_L (u+u')_t(\o)$.
\end{proof}

\subsection{Punctual differentiability}

When $u \in C^{1,2}(\Th, \dbR)$, it is immediately seen that $(\cL u_t(\o), \pa_\o u_t(\o)) \in cl(\ul\cJ_L u_t(\o))$ for $L\ge L_0$. Moreover, similar to \cite{ETZ1}, and also combining the arguments in Proposition \ref{prop:equiv}, one can easily show that the following are equivalent:

$\bullet$ $u$ is a classical subsolution at $(t,\o)$;

$\bullet$ $u$ is a viscosity subsolution at $(t,\o)$.

\ms
Following Caffarelli and Cabre \cite{CaffarelliCabre}, we introduce a notion of differentiation which is weaker than the path derivatives and will be crucial for the proof of our main comparison result.

\begin{defn}
\label{defn:punctual}
Let $\f\in \dbL^0(\dbF)$. We say $\f$ is $\cP_L$-punctually $C^{1,2}$ at $(t,\o)$, if
 \beaa
 \cJ_L\f_t(\o)
 :=
 \text{\rm cl}\Big(\ul\cJ_L\f_t(\o)\Big)
 \cap
 \text{\rm cl}\Big(\ol\cJ_L\f_t(\o)\Big)
 &\neq&
 \emptyset.
 \eeaa
\end{defn}

The following result is straightforward.
\begin{prop}\label{prop:punctual}
Let $u\in C^0(\Theta,\dbR)$.

\no {\rm (i).}\q If $u\in C^{1,2}(\Th,\dbR)$, then $u$ is $\cP_L$-punctually $C^{1,2}$ at all $(t,\o)$ with $(\cL u_t(\o), \pa_\o u_t(\o)) \in \cJ_L u_t(\o)$;

\no {\rm (ii).}\q If $u$ is $\cP_L$-punctually $C^{1,2}$ at  $(t,\o)$  and is a $\cP_L$-viscosity solution (resp. subsolution, supersolution) of the path-dependent PDE (\ref{PPDE}) at $(t,\o)$, then for any $(\a,\b)\in\cJ_L u_t(\o)$  we have
 \beaa
-\a-F\left(t,\o,u_t(\o),\si^T_t(\o)\b\right)
&=~~\mbox{(resp. $\le$, $\ge$)}&
0.
 \eeaa
\end{prop}

\section{Comparison result}\label{sec:present_mainresults}

We first introduce some notations for appropriate spaces.

\no $\bullet$ $~\dbS^{t,\o}_2 :=\big\{ Y\in \dbL^0(\dbF): Y~\mbox{is continuous in time, $\dbP_{\si^{t,\o}}$-a.s. and}~  \dbE^{\dbP_{\si^{t,\o}}} \big[ \sup_{0\le s\le T-t} |Y_s|^2 \big] <\infty\big\}$;

\ms

\no $\bullet$ $~\dbS^2 := \dbS^{0,0}_2$;

\ms 
\no $\bullet$   $~C^0_2(\Th) := \big\{ u\in C^0(\Th): u^{t,\o} \in \dbS^{t,\o}_2 ~\mbox{for all}~(t,\o)\in\Theta\big\}$;  

\ms

\no $\bullet$ $~\dbH^2 := \big\{Z\in \dbL^0(\dbF, \dbR^d) :  \dbE^{\dbP_{\si}} \big[\int_0^{T} |\si^T_s Z_s|^2ds \big] <\infty\big\}$;

\ms

\no $\bullet$ $~\dbI^2 := \big\{K\in \dbS^2 : K~\mbox{is increasing,  $\dbP_{\si}$-a.s. and}~  K_0=0\big\}$;

\no In particular, it follows from Assumption \ref{assum:si} and standard estimates for SDEs that $\si\in C^0_2(\Th)$.

\subsection{Main result}

The main focus of this paper is the following comparison result. 

\begin{thm}\label{thm:comparison}
Let Assumption \ref{assum:F} hold true, and $u,v\in C^0_{2}(\Theta)$ be $\cP_L$-viscosity subsolution and  supersolution, respectively, of PPDE (\ref{PPDE}) for some $L\geq L_0$. If $u_{T}\leq v_{T}$ on $\O$, then $u\le v$ on $\Theta$.
\end{thm}

A similar result in the case of $\si=I_d$ was proved in \cite{EKTZ}. Their proof is based on the construction of a regular approximation of the BSDE representation of the solution. Also, the comparison result in the fully nonlinear case addressed in \cite{ETZ2} is crucially based on an approximation by finite-dimensional partial differential equations induced by conveniently freezing the path-dependency. With this approximation in hand, the comparison result is proved by building on the corresponding classical results in the PDE literature.

The main contribution of this paper is to provide an alternative proof which does not rely on any representation of the solution, and which does not appeal to the corresponding PDE literature. We also observe that the comparison result of Theorem \ref{thm:comparison} allows for a random and possibly degenerate diffusion coefficient $\si$. Our proof of the comparison result is new, and is even relevant in the Markovian case which reduces to a PDE in a finite-dimensional space. Notice that in the last context, any test function $\phi(t,x)$ which is pointwise tangent from below to a function $f(t,x)$ at point $(t^*,x^*)$ induces a test process $\varphi_t(\omega):=\phi(t,\omega_t)$ which lies in $\ul\cA_Lu_{t^*}(\o^*)$ with $u_t(\o):=f(t,\o_t)$, whenever $\o^*_{t^*}=x^*$. In general, the opposite direction is not true, even for a Markovian test process $\varphi_t(\o)=\varphi(t,\o_t)$ in $\ul\cA_Lu_t(\o)$. This shows that our definition of viscosity solutions involves a larger class of test function than the standard Crandall-Lions notion of viscosity solutions in finite-dimensional spaces. Consequently, the comparison result has more chances under our definition, and we may hope to have an easier proof. We believe that the present proof achieves this goal. This is definitely true in the linear case which is isolated in  Subsection \ref{sect:heat}.

\subsection{Martingale representation and optimal stopping problem}\label{subsec:mgrep_optimalstop}

In this subsection, we state the results of the martingale representation under $\dbP_\si$ and the related optimal stopping problem, which is the key stone for our comparison principle of viscosity solutions. We report the corresponding proofs in Appendix so that the readers may have a clear perspective of the whole paper.

\begin{thm}[Martingale representation]\label{thm:mrt}
$\dbP_\si$ satisfies the martingale representation property. That is, for any $\xi\in \dbL^2(\cF_T, \dbP_\si)$, there exists unique $Z\in \dbH^2$ such that
\beaa
 \xi = \dbE^{\dbP_\si}[\xi] + \int_0^T Z_t \cd dB_t,\q\dbP_\si\mbox{-a.s.}
\eeaa 
\end{thm}

\begin{cor}
\label{cor:mrt}
Let $\l\in \dbL_L(\dbF)$ and $M\in \dbS^2$. Then $M$ is a $\dbP_{\si,\l}$-martingale if and only if there exists $Z\in \dbH^2$ such that
\beaa
d M_t = Z_t \cd \Big[dB_t -\si_t \l_t dt\Big],\q\dbP_\si\mbox{-a.s.}
\eeaa
\end{cor}

Let $\ch\in\cT^+$ and $X\in\dbL^0(\dbF)$ be a process with continuous sample paths. Consider the optimal stopping problem under dominated nonlinear expectation:
 \bea\label{Voptimal}
 V_t(\omega)
 :=
 \sup_{\tau\in\cT}
 \overline{\cE}^{t,\o}_L\big[X^{t,\omega}_{\tau\wedge (\ch^{t,\o}-t)}\big],
 &\mbox{for all}&
 (t,\omega)\in\Theta.
 \eea

\begin{thm}[Optimal stopping problem]
\label{thm:optimalstop_new}
Let $L>0$ and $X\in\dbL^0(\dbF)$ such that $X_{\cdot\we\ch}\in\dbS^2$. Then, there exists an $\dbF-$adapted and $\dbP_\si$-a.s. continuous process $Y$ satisfying:
\\
{\rm (i)}\q there exists $\t^*\in\cT_\ch$ such that $\t^*=\inf\{t:Y_t=X_t\}$, $\dbP_\si$-a.s. and $Y_0=\ol\cE_L[X_{\t^*}]$; 
\\
{\rm (ii)}\q for all $\tau\in\cT_\ch$, we have $Y_{\tau}=V_{\tau}$, $\dbP_\si-$a.s.; in particular, $Y_0=V_0$;
\\
{\rm (iii)}\q there exist $\dbP^*\in\cP_L$,   $\dbP^*$-martingale $M$ starting from $0$, and $K\in\dbI^2$ such that
 \beaa
 &Y=Y_0+M-K
 ~~\mbox{and}~~
 \int(Y-X)dK=0,~~\dbP_\si-\mbox{a.s.}&
 \eeaa
\end{thm}

\begin{defn}[Snell envelop]
\label{defn-Snell}
The process $Y$ introduced in Theorem \ref{thm:optimalstop_new} is called  a Snell envelop of the stopped process $X_{\cdot\we\ch}$, and denote $\mbox{{\rm Snell}}(X_{\cdot\we\ch}):=Y$. The stopping time $\t^*\in\cT_\ch$ is called  an optimal stopping rule.
\end{defn}

\subsection{Pathwise semimartingales}
In this subsection, let $u\in \dbL^0(\dbF)$ such that all the (nonlinear) expectations involved below exist. Similar to standard semimartingale under a fixed probability measure $\dbP$, we say $u$  is an $\ol\cE_L$-submartingale (resp.  supermartingale)  if, for any $t$ and any $\t\in \cT$ such that $\t\ge t$,
\bea
\label{cELsub}
u_t  \le   (\mbox{resp.}~\ge)~  \ol\cE_L[u_\t |\cF_t] := \esup_{\dbP\in \cP_L} \dbE^\dbP[u_\t |\cF_t],\q \dbP_\si\mbox{-a.s.}
\eea
Notice that viscosity solutions are pathwise defined. We extend the above notion in a pathwise manner.
\begin{defn}
\label{defn:semimg}
\no {\rm (i)} We say $u$ is a pathwise $\dbP_{\si}$-submartingale (resp. supermartingale) if 
\beaa
u_t(\o) \le (\mbox{resp.}~\ge)~ \dbE^{\dbP_{\si^{t,\o}}}[u^{t,\o}_\t]\q\mbox{for any $(t,\o)\in\Th$ and  $\t\in \cT_{T-t}$}.
\eeaa

\no {\rm (ii)} We say $u$ is a  pathwise $\ol\cE_L$-submartingale (resp. supermartingale) if 
\beaa
u_t(\o) \le (\mbox{resp.}~\ge) ~\ol\cE^{t,\o}_L[u^{t,\o}_\t]\q\mbox{for any $(t,\o)\in\Th$ and $\t\in \cT_{T-t}$}.
\eeaa
\end{defn}

\begin{rem}
\label{rem-semimg}
{\rm By Proposition \ref{prop:rcpd} and definition of r.c.p.d., it is clear that a  pathwise $\dbP_\si$-submartingale (resp. supermartingale) is a $\dbP_\si$-submartingale (resp. supermartingale).
}
\end{rem}

\begin{prop}
\label{prop:cEsub}
 Assume $u\in \dbS^2$ is a pathwise $\ol\cE_L$-submartingale. Then,
 
\no {\rm (i).}\q $u$ is an $\ol\cE_L$-submartingale;

\no {\rm (ii).}\q there exists $\dbP^*\in\cP_L$ such that $u$ is a $\dbP^*$-submartingale.
\end{prop}

\subsection{A fundamental lemma}

The following result shows how to find a point of tangency in mean. This replaces the local compactness argument in the standard Crandall-Lions theory of viscosity solutions.

\begin{lem}\label{lem:maximumpoint}
Assume $u\in\dbL^0(\dbF)$ satsfying $u_{\cdot\we\ch}\in\dbS^2$ and $u_0 > \ol\cE_L[u_{\ch}]$ for some  $\ch\in\cT^+$. Then there exists $\o^* \in \O$ and $t^*<\ch(\o^*)$ such that $0\in\underline\cA_Lu_{t^*}(\omega^*)$.
\end{lem}

\proof Define the optimal stopping problem $V$ by \reff{Voptimal} with $X := u$. Let $\t^*\in \cT_\ch$ be the optimal stopping rule. Since by Theorem \ref{thm:optimalstop_new} (i) and (ii) we have
\beaa
\ol \cE_L[u_{\t^*}] = V_0 \ge u_0 > \ol \cE_L [ u_\ch] 
&\mbox{and}& 
\dbP_\si\big[u_{\t^*} = V_{\t^*}\big] = 1,  
\eeaa
and it follows that $\dbP_\si\big[u_{\t^*} = V_{\t^*}, \t^* <\ch\big] > 0$, then there exists $\o^* \in \O$ such that $t^* := \t^*(\o^*) < \ch(\o^*)$ and $u_{t^*}(\o^*) = V_{t^*}(\o^*)$. By the definition of $V$ and $\ul\cA_L u$, this means that $(t^*, \o^*)$ is the desired point.
\qed

As a direct application of the lemma above, we obtain the comparison result for the heat equation in the next subsection.

\subsection{Comparison result for the heat equation}
\label{sect:heat}

In this subsection, we consider equations with nonlinearity $F=0$, i.e.
\bea\label{heat}
-\cL u(t,\o)=0
&t<T,&
\o\in\O.
\eea
Our objective is to provide an easy proof of the comparison result of Theorem \ref{thm:comparison} which requires standard tools from stochastic analysis. For simplicity, we specialize the comparison Theorem \ref{thm:comparison} to the case $L=0$, and call the corresponding viscosity solution as $\dbP_\si$-viscosity solution. We emphasize that the set of test processes is the largest possible with $L=0$.

\begin{thm}\label{thm:submart}
For a process $u\in C^0_2(\Th)$, the following are equivalent:
\\
{\rm (i)}\q $u$ is a pathwise $\dbP_\si$-submartingale (resp. supermartingale);
\\
{\rm (ii)}\q $u$ is $\dbP_\si$-viscosity subsolution (resp. supersolution) of the path-dependent heat equation \eqref{heat}.
\end{thm}

\proof  (i) $\Longrightarrow$ (ii): Assume to the contrary that, for some $(t,\omega)\in[0,T)\times\O$ and $\varphi\in\underline{\cA}_0 u_t(\omega)$ with localizing time $\ch\in\cT^+$,  $-c:=\cL\f_t(\o) <0$. Without loss of generality, we assume that $(t,\o)=(0,0)$. Note that 
 \beaa
 (\varphi-u)_0 \le \dbE^{\dbP_\si}\big[(\varphi-u)_{\tau}\big]
 &\mbox{for all}&
 \tau\in\cT_\ch.
 \eeaa
Denote $\t:= \inf\{t: \cL\f_t \ge -{c\over 2}\}\we\ch\in \cT^+$. Then, by (ii), we obtain the following desired contradiction:
 \beaa
 0
 \;\ge\;
 u_0 - \dbE^{\dbP_\si}\big[u_\t\big]
 &\ge&
 \varphi_0 - \dbE^{\dbP_\si}\big[\varphi_\tau\big]
 \;=\;
 \dbE^{\dbP_\si}\Big[-\int_0^\t
                  \cL\f_sds 
           \Big]
  \;\ge\; {c\over 2} \dbE^{\dbP_\si}[\t]>0.
 \eeaa

(ii) $\Longrightarrow$ (i): First, denote $u^{\e}_t(\o):=u_t(\o)+\e t$. It is easy to verify that $u^{\e}$ is a $\dbP_\si$-viscosity subsolution to the following equation:
 \beaa
 -\cL u^{\e}_t(\o)+\e
 &\le&
 0.
 \eeaa
We now show that $u^{\e}$ is a pathwsie $\dbP_\si$-submartingale. Suppose to the contrary that there exists a point $(t,\o)$ at which the supermartingale property fails, and set $(t,\o)=(0,0)$ without loss of generality. Then, there exists a stopping time $\ch\in\cT^+_T$ such that $u^{\e}_0>\dbE^{\dbP_\si}[u^{\e}_\ch]$. By Lemma \ref{lem:maximumpoint}, there exists $(t^*,\o^*)$ such that  $0\in\ul{\cA}_0 u^{\e}_{t^*}(\o^*)$, and it follows from the $\dbP_\si$-viscosity subsolution property of  $u^{\e}$ that $\e\le 0,$ which is the required contradiction.

Hence, $u^{\e}$ is a pathwise $\dbP_\si$-submartingale, namely $u_t(\o) +\e t \le \dbE^{\dbP^{\si^{t,\o}}}[u^{t,\o}_\t + \e (\t+t)]$ for all $\t\in \cT_{T-t}$. Send $\e\to 0$, we obtain immediately that $u$ is a  a pathwise $\dbP_\si$-submartingale.
\qed

Theorem \ref{thm:submart} leads immediately to the comparison result. 

\begin{thm}\label{thm:heat}
 Let $u,v\in C^0_2(\Th)$ be $\dbP_\si$-viscosity subsolution and $\dbP_\si$-viscosity supersolution, respectively,  of path dependent heat equation (\ref{heat}). If  $u_T\le v_T$ on $\O$, then $u\le v$ on $\Theta$.
\end{thm}

\begin{rem}
{\rm By Theorem \ref{thm:submart} we see that our notion of $\dbP_\si-$viscosity solution reduces to the notion of stochastic viscosity solution introduced by Bayraktar and Sirbu \cite{BayraktarSirbu1,BayraktarSirbu2} in the Markovian case.
}
\end{rem}

\begin{rem}
{\rm (i) Theorem \ref{thm:submart} also provides the unique solution of the heat equation. Indeed it implies that a pathwise $\dbP_\si$-martingale is a viscosity solution. Since the final value is fixed by the boundary condition $\xi$, we are naturally lead to the candidate solution $u(t,\o):=\dbE^{\dbP_{\si^{t,\o}}}\big[\xi^{t,\o}\big]$, $(t,\o)\in\Theta$. Therefore, if this process is in $C^0_2(\Th)$, it is the unique viscosity solution of the heat equation.

(ii) For the heat equation, we can in fact prove the comparison principle without requiring the continuity (in $\o$) of the viscosity semi-solutions.}
\end{rem}

\subsection{Partial comparison}

We next return to the general semilinear PPDE \reff{PPDE}. The following partial comparison result, as in \cite{EKTZ} and \cite{ETZ1}, is a crucial step for our proof of the comparison result.

\begin{prop}\label{prop:partialcomparison}
In the setting of Theorem \ref{thm:comparison}, if in addition $v\in C^{1,2}(\Theta)$, then  $u\le v$ on $\Theta$.
\end{prop}

\proof
We report the proof from \cite{EKTZ} for completeness. First, by possibly transforming the problem to the comparison of $\tilde u_t:=e^{\lambda t}u_t$ and $\tilde v_t:=e^{\lambda t}v_t$, it follows from the Lipschitz property of the nonlinearity $F$ in $y$ that we may assume without generality that $F$ is decreasing in $y$.

Suppose to the contrary that $c:=(u-v)_t(\o)>0$ at some point $(t,\o)\in[0,T)\times\O$. Without loss of generality assume $(t,\o) = (0,0)$.  Let $c_0:=\frac{c}{2T}$, and define $X_s:=(u-v)^+_{s}+c_0s$, $s\in[0,T]$. Clearly $X\in  C^0_{2}(\Theta)$. Since $(u-v)_T\le 0$, it follows that $X_0>\ol\cE_L[X_T]$. By Lemma \ref{lem:maximumpoint}, we may find a point $(t^*,\o^*)$ such that $t<T$ and $0\in\ul\cA_LX_{t^*}(\o^*)$. In particular, this implies that
 \beaa
 -(u-v)^+_{t^*}(\o^*)-c_0t^*
 &\le&
 \ul\cE_L\big[-\big\{(u-v)^+\big\}^{t^*,\o^*}_{T-t^*}-c_0T\big]
 \;=\;
 -c_0T,
 \eeaa
and thus  $(u-v)^+_{t^*}(\o^*)\ge c_0(T-t^*)>0$. Therefore, $(u-v)^+_{t^*}(\o^*)= (u-v)_{t^*}(\o^*) > 0$. Then, since $(u-v)^+\ge u-v$, we deduce from $0\in\ul\cA_LX_{t^*}(\o^*)$ that
 \beaa
 (\varphi-u)_{t^*}(\o^*)
 \le
 \ul\cE_L\big[(\varphi-u)^{t^*,\o^*}_{\tau}\big]
 &\mbox{for all}~\tau\in\cT_{T-t^*},~\mbox{where}&
 \varphi_s(\o):= v_s(\o)-c_0s.
 \eeaa
Since $v\in C^{1,2}(\Theta)$, this means that $\f\in\ul\cA_Lu_{t^*}(\o^*)$.  Note that $\cL\f = \cL v - c_0$ and $\pa_\o \f = \pa_\o v$. Then, since $u$ is a  viscosity subsolution  and $v$ is a  classical supersolution, we deduce that
 \beaa
 0
 &\ge&
 \{-\cL \varphi-F(.,u, \si^T\partial_\o \varphi)\}(t^*,\o^*)
 \\
 &=&
 c_0+\{-\cL v-F(.,u, \si^T\partial_\o v)\}(t^*,\o^*)
 \\
 &\ge&
 c_0+\{F(.,v, \si^T\partial_\o v)
           -F(.,u, \si^T\partial_\o v)\}(t^*,\o^*)
 \;\ge\;
 c_0,
 \eeaa
where the last inequality follows from the non-increase of  $F$ in $y$ and the fact that $u_{t^*}(\o^*)\ge v_{t^*}(\o^*)$. Since $c_0>0$, this is the required contradiction.
\qed

\subsection{Punctual differentiability of viscosity semi-solutions}

We first extend part of Theorem \ref{thm:submart} to this case.

\begin{lem}\label{v a.s. PC2}
Let  Assumption \ref{assum:F} hold, and for some $L\ge L_0$, $u \in C^0_2(\Th)$ be an $L$-subsolution of PPDE \reff{PPDE}. Then, the process $\hat u:= u+\int_0^. (L_0 |u_s| + F^0_s+1) ds$ is a  pathwise $\ol\cE_L$-submartingale.
\end{lem}

\begin{proof} 
Suppose to the contrary that $\hat{u}_t(\o)>\ol\cE^{t,\o}_L[\hat{u}^{t,\o}_\ch]$ for some $(t,\o)\in[0,T)\times\O$ and $\ch\in\cT^+_{T-t}$. Then, it follows from Lemma \ref{lem:maximumpoint} that there exist $\o^*\in \O$ and $t^* \in [t, t+\ch(\o^*))$  such that $0\in\ul\cA_L\hat{u}_{t^*}(\o^*)$, that is, there exists $\ch'\in\cT^+_{T-t^*}$ such that
\beaa
-\hat u_{t^*}(\o^*)\le \ul\cE_L^{t^*,\o^*}\Big[-\hat u^{t^*,\o^*}_{\t}\Big] &\mbox{for all}& \t\in \cT_{\ch'}.
\eeaa
Rewriting it we have
\beaa
-u_{t^*}(\o^*)\le \ul\cE_L^{t^*,\o^*}\Big[ \f_\t   - u^{t^*,\o^*}_{\t}\Big]~ \mbox{for all}~ \t\in \cT_{\ch'}, &\mbox{where}& \f_t := -  \int_0^t \big(L_0|u_s| + (F^0)_s+1\big)ds.
\eeaa
Clearly  $\f\in C^{1,2}(\Th)$ with $\cL\f_{t^*}(\o^*) = -L_0|u_{t^*}(\o^*)| -F^0_{t^*}(\o^*)-1$ and $\pa_\o \f_{t^*}(\o^*) = 0$. Then the above inequality implies that $\f\in \ul\cA_L u_{t^*}(\o^*)$.    Now by the viscosity subsolution property of $u$ and Assumption \ref{assum:F}, we have
\beaa
0 &\ge& -\cL\f_{t^*}(\o^*) - F_{t^*}(\o^*, u_{t^*}(\o^*), \si^T_{t^*}(\o^*) \pa_\o \f_{t^*}(\o^*)) \\
&=&  L_0|u_{t^*}(\o^*)| +F^0_{t^*}(\o^*)+1 -  F_{t^*}(\o^*, u_{t^*}(\o^*), 0)  \ge F^0_{t^*}(\o^*) +1 -  F_{t^*}(\o^*, 0, 0) \ge 1,
\eeaa
contradiction.
\end{proof} 

Unlike the heat equation case, the above property and the corresponding statement for a viscosity supersolution $v$ does not lead to the comparison principle directly. Our main idea is the following punctual differentiability of $u$. 

\begin{prop}\label{prop: submartingale PC2}
 Assume  $u$ is a $\dbP_\si-$semimartingale with decomposition: $du_t = Z_t \cd dB_t + dA_t$, where $Z\in \dbH^2$ and $A\in \dbL^0(\dbF)$ is continuous and has finite variation, $\dbP_\si$-a.s. Then there exist a Borel set $\dbT^u\subset [0, T]$ and $\O^u_t \in \cF_t$ for each $t\in \dbT^u$ such that, for any $L>0$,
\bea
\label{TOu}
Leb(\dbT^u) = T, ~\dbP_\si(\O^u_t) = 1,~~ \mbox{and $u$ is $\cP_L$-punctually $C^{1,2}$ at $(t,\o)$ for all $t\in \dbT^u$, $\o\in \O^u_t$.}
\eea
\end{prop}

\proof Denote 
\beaa
\zeta_t := \limsup_{0\downarrow  h \in \dbQ}{1\over h} \int_t^{t+h} |\si_s^TZ_s - \si_t^TZ_t|ds,~ \dot{A}^+_t := \limsup_{0\downarrow  h\in \dbQ}{1\over h}[A_{t+h} - A_t],~ \dot{A}^-_t := \liminf_{0\downarrow  h\in \dbQ}{1\over h}[A_{t+h} - A_t].
\eeaa
Note that 
the processes $\zeta, \dot{A}^+$, and $\dot{A}^-_t$ are $\dbF^+$-measurable (with possible values $\infty$ and $-\infty$).  Denote
\bea
\label{Th0}
\left.\ba{lll}
\O_0 := \Big\{\o\in \O: \int_0^T |\si_t^TZ_t(\o)|dt <\infty, \mbox{and $A$ is continuous and has finite variation on $[0, T]$}\Big\};\\
\Th_0 := \Big\{(t,\o)\in [0, T) \times \O: \zeta_t(\o) = 0,    \dot{A}^+_t(\o) =  \dot{A}^-_t(\o) \in \dbR\Big\} \in \cB\big([0,T]\big)\times \cF_T,
\ea\right.
\eea
Then $\dbP_\si(\O_0)=1$, and, by the Lebesgue differentiation theorem (see e.g. \cite{Rudin} Theorem 7.7, p. 139), 
\beaa
\mbox{Leb}\Big[t:  (t,\o) \in \Th_0\Big] = T &\mbox{for all}& \o\in \O_0.
\eeaa
Applying Fubini Theorem there exists $\dbT^u \subset [0, T]$ such that
\bea
\label{dbT}
\mbox{Leb}[\dbT^u] = T &\mbox{and}& \dbP_\si[\O_t^1] =1 ~\mbox{for all}~ t\in \dbT^u,~ \mbox{where}~ \O^1_t := \{\o\in \O: (t,\o) \in \Th_0\}.
\eea
Note that $\O^1_t \in \cF_{t+}\subset \cF^*_t$, thanks to Proposition \ref{prop:01} in Appendix. 
Moreover, for any $t\in \dbT^u$,  by Proposition \ref{prop:rcpd} one can easily see that there exists $\O^2_t\in \cF_t$ such that  
\bea
\label{O2t}
\dbP_\si[\O^2_t] = 1 &\mbox{and}& d u^{t,\o}_s = Z^{t,\o}_s \cd dB_s + dA^{t,\o}_s, 0\le s\le T-t, ~\dbP_{\si^{t,\o}}\mbox{-a.s. for all}~\o\in \O^2_t.
\eea
Now define $\O_t := \O^1_t \cap \O^2_t\cap \O_0\in\cF^*_t$ for all $t\in \dbT^u$, then we may find $\O^u_t\subset\O_t$ such that
\bea
\label{Ot}
\O^u_t \in \cF_t,\q \dbP_\si[\O^u_t] = 1,\q \mbox{for all}~t\in \dbT^u.
\eea

Define $ \dot{A}_t(\o) :=  \dot{A}^+_t(\o) =  \dot{A}^-_t(\o)$ for $(t,\o)\in \Th_0$.  We claim that $(\dot{A}_t(\o), Z_t(\o))\in \cJ_L u_t(\o)$ for all $t\in \dbT^u$, $\o\in \O^u_t$ and $L>0$.
Without loss of generality, we shall only show that 
\bea
\label{AZJ}
(\dot{A}_t(\o)+\e, Z_t(\o))\in \ul\cJ_L u_t(\o) &\mbox{for any}&\e>0. 
\eea
Indeed, fix $t\in \dbT^u$ and $\o\in \O^u_t$. First, since $A(\o)$ is continuous, we have
\beaa
\lim_{h\downarrow 0}{1\over h} \int_t^{t+h} |\si_s^TZ_s(\o) - \si_t^TZ_t(\o)|ds =0,~  \lim_{h\downarrow 0}{1\over h}[A_{t+h}(\o) - A_t(\o)] = \dot{A}_t(\o).
\eeaa
Next, set $\d:= {\e\over 2L(1+|Z_t(\o)|}$. By Lemma \ref{lem:hit} in Appendix, there exists $\ch\in \cT_{T-t}$ such that
\beaa
& \ch = \inf\Big\{s> 0: \int_0^s |(\si^TZ)^{t,\o}_r - (\si^TZ)_t(\o)|dr \ge \d s, ~\mbox{or}~ |\si^{t,\o}_s - \si_t(\o)|\ge \d, &\\
& \mbox{or}~ A^{t,\o}_{s}-A_t(\o) \ge (\dot{A}_t(\o) + {\e\over 2}) s\Big\} \; \wedge \; (T-t),\q \dbP_{\si^{t,\o}}\mbox{-a.s.} & 
\eeaa
By \reff{Th0} we see that $\ch>0$ and thus $\ch \in \cT^+_{T-t}$. For any $\l\in \dbL_L(\dbF)$ and $\t\in \cT_\ch$, by \reff{O2t} we have
\beaa
&&\dbE^{\dbP_{\si^{t,\o},\l}}\Big[u^{t,\o}_\t - Q^{ \dot{A}_t(\o)+\e, Z_t(\o)}_\t \Big] - u_t(\o) \\
&=&  \dbE^{\dbP_{\si^{t,\o},\l}}\Big[u^{t,\o}_\t  - u^{t,\o}_0-  ( \dot{A}_t(\o)+\e) \t -  Z_t(\o) \cd B_\t \Big] \\ 
 &=& \dbE^{\dbP_{\si^{t,\o},\l}}\Big[ \int_0^\t [Z^{t,\o}_s - Z_t(\o)] \cd dB_s  + (A^{t,\o}_\t - A^{t,\o}_0) -( \dot{A}_t(\o)+\e) \t  \Big] \\
&=& \dbE^{\dbP_{\si^{t,\o},\l}}\Big[ \int_0^\t [Z^{t,\o}_s - Z^{t,\o}_0] \cd (\si^{t,\o}_s \l_s) ds    + (A^{t,\o}_\t - A^{t,\o}_0) -( \dot{A}_t(\o)+\e) \t  \Big]\\
& \le&  \dbE^{\dbP_{\si^{t,\o},\l}}\Big[L\int_0^\t |(\si^TZ)^{t,\o}_s - (\si^TZ)_t(\o)|ds  + L |Z_t(\o)| \int_0^\t |\si^{t,\o}_s - \si_t(\o)|ds \\
&&\qq + (A^{t,\o}_\t - A^{t,\o}_0) -( \dot{A}_t(\o)+\e) \t  \Big]\\
&\le& \dbE^{\dbP_{\si^{t,\o},\l}}\Big[ L\d \t  + L|Z_t(\o)|\d\t + (\dot{A}_t(\o) + {\e\over 2}) \t -( \dot{A}_t(\o)+\e) \t  \Big] = 0,
\eeaa
Then \reff{AZJ} follows from the arbitrariness of $\l$ and $\t$.
\qed

\subsection{Comparison result for general semilinear PPDEs}

We are now ready for the key step for the proof of Theorem \ref{thm:comparison}. We observe that this statement is an adaptation of the approach of Caffarelli and Cabre \cite{CaffarelliCabre} to the comparison in the context of the standard Crandall-Lions theory of viscosity solutions in finite dimensional spaces. See their Theorem 5.3 p45.

\begin{prop}\label{prop:wsubsol}
Let Assumption \ref{assum:F} hold true, and $u,v\in C^0_{2}(\Theta)$ be $\cP_L$-viscosity subsolution and  supersolution, respectively, of PPDE (\ref{PPDE}) for some $L\geq L_0$.  Then, $w:=u-v$ is an $L$-viscosity subsolution of
\begin{equation}\label{equation w}
-\cL w(t,\o)-L|w_t(\o)|-L|\si^T_t(\o)\pa_\o w_t(\o)|\leq 0.
\end{equation}
\end{prop}

Before we prove this proposition, we use it to complete the proof of Theorem \ref{thm:comparison}.

\vspace{2mm}

\no {\bf Proof of Theorem \ref{thm:comparison}}\quad
By Proposition \ref{prop:wsubsol}, functional $u-v$ is a $\cP_L-$viscosity subsolution of PPDE (\ref{equation w}). Clearly, $0$ is a classical supersolution of the same equation. Since $(u-v)_T\le 0$, we conclude from the partial comparison Proposition \ref{prop:partialcomparison} that $u-v\le 0$ on $\Theta$.
\qed

\bs

\no {\bf Proof of Proposition \ref{prop:wsubsol}}\quad
Without loss of generality, we only check the viscosity subsolution property at
$(t,\o)=(0,0)$. For an arbitrary $(\a,\b)\in\ul\cJ_Lw_0$, we want to show that
 \bea\label{requiredu-vsubsol}
 -\a-L|w_0|-L|\si^T_0\b|
 &\le&
 0.
 \eea
 
\no {\bf 1.}  By definition, there exists $\ch\in\cT^+$ such that
 \beaa
 w_{0}
 &=&
 \max_{\t\in\cT_\ch}\ol\cE_L\big[(w-Q^{\a,\b})_{\t}\big].
 \eeaa
 Fix $\d>0$.  By otherwise choosing a smaller $\ch$, we may assume without loss of generality that
 \bea
 \label{Hsmall}
 |\f_t - \f_0|\le \d &\mbox{for}& \f = B, \si, u, v.
 \eea
Recall Definition \ref{defn-Snell} and  introduce the processes
\beaa
X := w-Q^{\a+\d,\b}
\q \mbox{and} \q
Y := \mbox{Snell}(X_{\cdot\we\ch}).
\eeaa
Clearly, since $\d>0$,
 \bea\label{mmhat}
 \ol\cE_L\left[X_{\ch}\right] < w_0 = X_{0} \le Y_0 ~~&\mbox{and}& Y_{\ch} = X_{\ch}, ~~\dbP_\si\mbox{-a.s.}
 \eea
Then, it follows from \eqref{mmhat} and Theorem \ref{thm:optimalstop_new} (iii) that there exists $\dbP^*\in\cP_L$ and $K\in\dbI^2$ such that
 \beaa
 0&>&  \ol\cE_L\left[Y_{\ch}-Y_0\right] \ge \dbE^{\dbP^*}\left[Y_{\ch}-Y_0\right] = - \dbE^{\dbP^*}\left[K_{\ch}\right] = -   \dbE^{\dbP^*}\left[\int_0^{\ch} \1_{\{Y_t=X_t\}} dK_t\right] .
\eeaa
We shall prove in Step 3 below that
 \bea\label{Claim1}
 K
 &\mbox{is absolutely continuous,}&
 \dbP_\si-a.s.
 \eea
Then, denoting by $\dot{K}$ the derivative of $K$ and noticing that $\dbP^*$ is equivalent to $\dbP_\si$, we deduce from the previous inequalities that:
\bea
 \label{LP>0}
 \dbE^{\dbP^*}\left[\int_0^{\ch} \1_{\{Y_t=X_t\}} \dot{K}_tdt\right] >0 &\mbox{and thus}& \mbox{Leb}\otimes\dbP_\si\big[ t< {\ch}, Y_t=X_t\big] > 0. 
\eea

Moreover, combining Lemma \ref{v a.s. PC2}, Remark \ref{rem-semimg}, and Proposition \ref{prop:cEsub}, we see that $\hat u$, and hence $u$, is a $\dbP_\si$-semimartingale.  Then by Proposition \ref{prop: submartingale PC2},   there exist measurable sets $\dbT^u\subset [0, T]$ and $\O^u_t \in \cF_t$ for each $t\in \dbT^u$ such that \reff{TOu} holds. Similarly, we may find $\dbT^v$ and $\O^v_t$ such that  \reff{TOu} holds for $v$ as well. Then \reff{LP>0} leads to:
 \beaa
 \mbox{Leb}\otimes\dbP_\si\big[ t\in[0,{\ch})\cap\mathbb{T}^u\cap\mathbb{T}^v, Y_t=X_t \big]
 &>&
 0,
 \eeaa
and thus there exists 
 \beaa
 t^*\in\mathbb{T}^v\cap\mathbb{T}^u
 &\mbox{such that}&
 \dbP_\si\big[t^*< \ch, Y_{t^*}= X_{t^*}\big]
 \;>\;
 0,
 \eeaa
which implies further that, recalling the $V$ defined in \reff{Voptimal} and Theorem \ref{thm:optimalstop_new} (i),
 \beaa
  \dbP_\si\Big[\O^u_{t^*}\cap\O^v_{t^*}\cap\{t^*<\ch, Y_{t^*}= X_{t^*}\}  \cap \{Y_{t^*} = V_{t^*}\} \Big]
 \;>\;
 0.
 \eeaa
Therefore, there exists $\omega^{*}\in \O$ such that
 \bea
 \label{to*}
\left.\ba{c}
\mbox{both $u$ and $v$ are $\cP_L$-punctually $C^{1,2}$ at $(t^*, \o^*)$,}\\
t^*<{\ch}(\o^*) ~\mbox{and}~ X_{t^*}(\o^*)  = \sup_{\t\in \cT} \ol\cE^{t^*,\o^*}_L\big[X^{t^*,\o^*}_{\t\we (\ch^{t^*,\o^*}-t^*)}\big].
 \ea\right.
 \eea
 
 \no{\bf 2.}   Let  $(\a^u,\b^u)\in\cJ_L u(t^*,\o^*)\subset cl(\ol J_L u_{t^*}(\o^*))$ and $(\a^v,\b^v)\in\cJ_L v(t^*,\o^*) \subset cl(\ul J_L v_{t^*}(\o^*))$. Then $(\a^u -\d,\b^u)\in \ol J_L u_{t^*}(\o^*)$ and $(\a^v+\d,\b^v)\in \ul \cJ_L v(t^*,\o^*)$. Apply Proposition \ref{new sum derv}, we have
 \bea
 \label{ab'}
 (\a', \b') \in \ol \cJ_L X_{t^*}(\o^*), &\mbox{where}&  \a':=  \a^u - \a^v-\a-3\d, ~ \b' := \b^u-\b^v-\b.
 \eea
 Choose $\l\in \dbL_L(\dbF)$ such that  $(\si^T)^{t^*,\o^*} \b' \cd \l = L|(\si^T)^{t^*,\o^*} \b'|$.  Then, for any $\e>0$, letting $\ch'\le \ch^{t^*,\o^*}-t^*$ be a common localizing time satisfying $|\si^{t^*,\o^*}_t - \si_{t^*}(\o^*)|\le \e$ for $0\le t\le \ch'$, we have
 \beaa
 X_{t^*}(\o^*) &\le& \ul\cE_L\Big[X^{t^*,\o^*}_{\ch'} - Q^{\a',\b'}_{\ch'}\Big] \le \dbE^{\dbP_{\si^{t^*,\o^*},\l}}\Big[X^{t^*,\o^*}_{\ch'} - Q^{\a',\b'}_{\ch'}\Big] \\
 &=& \dbE^{\dbP_{\si^{t^*,\o^*},\l}}\Big[X^{t^*,\o^*}_{\ch'} - \a' \ch' - \int_0^{\ch'} L|(\si^T)^{t^*,\o^*}_t \b'| dt\Big]\\
 &\le&  \ol\cE^{t^*,\o^*}_L[ X^{t^*,\o^*}_{\ch'}] -  (\a' + L|\si^T_{t^*}(\o^*) \b'| - L\e |\b'|) \dbE^{\dbP_{\si^{t^*,\o^*},\l}}[ \ch'].
 \eeaa 
 This, together with the optimality in \reff{to*}, implies that $\a' + L|\si^T_{t^*}(\o^*) \b'| - L\e |\b'|\le 0$. Since $\e>0$ is arbitrary, we obtain 
\beaa
\a'+ L|\si^T_{t^*}(\o^*) \b'|  \le 0.
\eeaa

Moreover, applying Proposition \ref{prop:equiv}, the semi-viscosity properties  of $u$ and $v$ lead to
 \beaa
 -\a^u - F_{t^*}(\o^*, u_{t^*}(\o^*), \si^T_{t^*}(\o^*) \b^u) \le 0,\q  -\a^v - F(_{t^*}(\o^*, v_{t^*}(\o^*), \si^T_{t^*}(\o^*) \b^v) \ge 0.
 \eeaa
 Then, recalling \reff{ab'} and by \reff{Hsmall},
 \beaa
 0 &\le& \a^u + F_{t^*}(\o^*, u_{t^*}(\o^*), \si^T_{t^*}(\o^*) \b^u)  -\a^v - F_{t^*}(\o^*, v_{t^*}(\o^*), \si^T_{t^*}(\o^*) \b^v) - \a' - L|\si^T_{t^*}(\o^*) \b'|\\
 &\le& \a+3\d + L |w_{t^*}(\o^*)| + L|\si^T_{t^*}(\o^*) \b|  \le \a+ L |w_0| + L|\si^T_0 \b|  + (3+ L + L|\b|) \d. 
 \eeaa
Now send $\d\to 0$, we  obtain  \reff{requiredu-vsubsol}.

 \ms
 
\no {\bf 3.} It remains to prove \eqref{Claim1}. By Proposition \ref{v a.s. PC2} and Remark \ref{rem-semimg}, we know the process $\hat u$  is an $\ol\cE_L$-submartingale.  Then it follows from Proposition \ref{prop:cEsub} and Corollary \ref{cor:mrt} that there exist $\l^u\in \dbL_L(\dbF)$ and $K^u\in \dbI^2$ such that 
\beaa
d \hat u_t =  Z^u_t \cd (dB_t - \si_t \l^u_t dt) + d K^u_t,\q \dbP_\si\mbox{-a.s.}  
\eeaa
This implies
\beaa
d u_t = Z^u_t \cd dB_t - (\si_t \l^u_t \cd Z^u_t - L_0 |u_t|-F^0_t -1) dt + dK^u_t,\q \dbP_\si\mbox{-a.s.}  
\eeaa
Similarly, for some  $\l^v\in \dbL_L(\dbF)$ and $K^v\in \dbI^2$,
\beaa
d v_t = Z^v_t \cd dB_t + (-\si_t \l^v_t \cd Z^v_t + L_0 |v_t|+F^0_t+1) dt - dK^v_t,\q \dbP_\si\mbox{-a.s.}  
\eeaa 
Thus, with certain appropriately defined processes $Z^X$, $\si^X$, and the $\l^*$ corresponding to $\dbP^*$,  
\bea
\label{dX}
dX_t = Z^X_t  \cd (dB_t - \si_t \l^*_tdt)  - \si^X_t dt + d(K^u_t + K^v_t)  ,\q \dbP_\si\mbox{-a.s.} 
\eea

Now for any $0\le s\le t\le T$, define $\t_s:=\inf\{t\geq s\we\ch: X_t=Y_t\}$. Recalling Theorem \ref{thm:optimalstop_new} (iii), we have  $K_{\t_s} = K_{s\we\ch}$, $\dbP_\si$-a.s. Then, by \reff{dX} we have
\beaa
&\dbE^{\dbP^*}\Big[K_{t\wedge \ch} - K_{s\wedge \ch} \Big|\cF_{s\wedge \ch}\Big] 
= \dbE^{\dbP^*}\Big[K_{t\wedge \ch} - K_{\t_s}\Big|\cF_{s\wedge \ch}\Big]
=  \dbE^{\dbP^*}\Big[ Y_{\t_s}-Y_{t\wedge \ch} \Big|\cF_{s\wedge \ch}\Big] 
\le \dbE^{\dbP^*}\Big[ X_{\t_s}-X_{t\wedge \ch} \Big|\cF_{s\wedge \ch}\Big] &
\\
 & = \dbE^{\dbP^*}\Big[\int_{\t_s}^{t\wedge \ch} (\si^X_r dr - dK^u_r - dK^v_r) \Big|\cF_{s\wedge \ch}\Big] 
\le  \dbE^{\dbP^*}\Big[\int_{s\wedge \ch}^{t\wedge \ch} |\si^X_r| dr  \Big|\cF_{s\wedge \ch}\Big].&
\eeaa
This implies that $dK_t \le |\si^X_t|dt$, $\dbP^*$-a.s. and hence also $\dbP_\si$-a.s.
\qed

\section{Existence}\label{sec:existence}

To construct a viscosity solution to a semilinear path-dependent PDE, we need to introduce BSDEs. Now  for any $(t,\o)\in\Th$,  $\t\in \cT_{T-t}$, and $\xi \in \dbL^2(\cF_{\t}, \dbP_{\si^{t,\o}})$, consider the following BSDE under $\dbP_{\si^{t,\o}}$:
\bea
\label{BSDE}
Y_s = \xi + \int_s^\t F^{t,\o}_r(B_\cd, Y_r, (\si^T)^{t,\o}_r Z_r) dr - \int_s^\t Z_r\cd dB_r,\q 0\le s\le \t, \dbP_{\si^{t,\o}}\mbox{-a.s.}
\eea
By Assumption \ref{assum:F} and Theorem \ref{thm:mrt}, additionally assuming that 
\beaa
\dbE^{\dbP_{\si^{t,\o}}}\Big[\int_0^{T-t}F^{t,\o}_s(B,0,0)^2ds\Big]<\infty
\q
\mbox{for all}~(t,\o)\in\Th,
\eeaa
one may easily prove  by standard arguments that the above BSDE admits a unique $\dbF$-measurable solution, denoted as  $(\cY^{t,\o}(\t, \xi), \cZ^{t,\o}(\t, \xi))$. Now, fix $\xi \in \dbL^0(\cF_T)$ such that $\xi^{t,\o}\in  \dbL^2(\cF_{T-t}, \dbP_{\si^{t,\o}})$ for any $(t,\o)\in\Th$, define
\bea
\label{urep}
u(t,\o) := \cY^{t,\o}_0(T-t, \xi^{t,\o}).
\eea 

\begin{thm}
\label{thm:existence}
Let Assumption  \ref{assum:F} hold true. Assume $F$ is continuous in $t$ and $u\in C^0_2(\Th)$. Then $u$ is an $L$-viscosity solution of PPDE \reff{PPDE}  for any $L\ge L_0$.
\end{thm} 
\proof 
Since $u\in C^0_2(\Th)$, together with standard arguments, \eqref{urep} implies the  
dynamic programming principle: given $(t, \o) \in \Th$ and $\t\in \cT_{T-t}$,
\bea
\label{DPP}
u(t, \o) = \cY^{t,\o}_{t}(\t, u^{t,\o}_\t).
\eea

Without loss of generality, we check only the viscosity subsolution property at $(0,0)$. Assume not, then there exists $\f\in \ul\cA_L u_0$ with localizing time $\ch$ such that $-c := \cL\f_0 + F_0(u_0, \si_0^T\pa_\o \f_0) <0$. By continuity there exists $\t \in \cT^+_\ch$  such that  $\cL\f_t + F_t(u_t, \si_t^T \pa_\o \f_t) \le -{c\over 2} $ for $0\le t\le \t$. Note that $u_t = \cY^{0,0}_t(\t, u_\t)$ and denote $Z_t := \cZ^{0,0}(\t, u_\t)$. Then, by \reff{DPP} and the functional It\^o formula \reff{Ito}, 
\beaa
[\f-u]_\t -[\f-u]_0 &=& \int_0^\t [\cL \f_t  + F_t(u_t, \si_t^T Z_t)] dt + \int_0^\t [\pa_\o \f_t - Z_t] \cd dB_t\\
&\le& \int_0^\t [-{c\over 2}   + F_t(u_t, \si_t^T Z_t) - F_t(u_t, \si_t^T \pa_\o \f_t)] dt + \int_0^\t [\pa_\o \f_t - Z_t] \cd dB_t\\
&=&   \int_0^\t \big[-{c\over 2}   -[\pa_\o \f_t - Z_t] \cd  \si_t \l_t\big] dt + \int_0^\t [\pa_\o \f_t - Z_t] \cd dB_t,~\dbP_\si\mbox{-a.s.}
\eeaa
where $\l \in \dbL_{L_0}(\dbF)$.  Note that $\dbP_{\si,\l}$ and $\dbP_{\si}$ are equivalent. This implies
\beaa
[\f-u]_\t -[\f-u]_0 &\le& -{c\over 2}\t    + \int_0^\t [\pa_\o \f_t - Z_t] \cd [dB_t - \si_t \l_t dt],~\dbP_{\si,\l}\mbox{-a.s.}
\eeaa
Thus, noting that $L\ge L_0$ and that $dB_t - \si_t \l_t dt$ is a $\dbP_{\si,\l}$-martingale,
\beaa
[\f-u]_0 \ge \dbE^{\dbP_{\si,\l}}\Big[[\f-u]_\t + {c\over 2}\t\Big] >  \dbE^{\dbP_{\si,\l}}\Big[[\f-u]_\t \Big] \ge \ul \cE_L\Big[[\f-u]_\t \Big] ,  
\eeaa
contradicting with the fact that $\f\in \ul\cA_L u_0$.
\qed

The following proposition gives a sufficient condition so that $u\in C^0_2(\Th)$. The proof follows from standard BSDE estimates, and thus is omitted.

\begin{prop}
If $F$ and $\xi$ are both uniformly continuous in $\o$ and $F$ is continuous in $t$, then $u\in C^0_2(\Th)$.
\end{prop}

\section{Appendix}\label{sec:appendix}

\subsection{Martingale representation}\label{subsec:mgrep}
 We start with a simple lemma. Recall \reff{Xsi} and denote $X:= X^\si$ for notational simplicity. 

\begin{lem}\label{lem:filtration XW}
For any $\eta\in \dbL^1(\cF^X_T, \dbP_0)$, we have
$
\dbE^{\dbP_0}[\eta|\cF^X_t]=\dbE^{\dbP_0}[\eta|\cF_t],~\dbP_0\mbox{-a.s.}
$
\end{lem}
\begin{proof} Denote $\cG_t:=\si\{B_s-B_t:s\geq t\}$.   Since $X$ is a strong solution, we see that $\cF^X_{t}\subset\cF_{t}$ and $\cF^X_T\subset\cF^X_t\vee\cG_t$. In particular, $\cF^X_{t}$ and $\cF_t$ are independent of $\cG_t$ under $\dbP_0$. Then, 
\beaa
\dbE^{\dbP_0}[\1_E\1_{E'}|\cF^X_t]=\1_E\dbP_0[E']=\dbE^{\dbP_0}[\1_E\1_{E'}|\cF_t], &\mbox{for any}& E\in \cF^X_t, E'\in \cG_t.
\eeaa
Now the result follows from the standard argument of monotone class theorem.
\end{proof}

We next establish the martingale representation property for $\dbP_\si$.

\no {\bf Proof of Theorem \ref{thm:mrt}} \q By standard approximation arguments, we may assume without loss of generality that $\xi$ is Lipschitz continuous in $\o$. Denote 
\beaa
u(t,\o) := \dbE^{\dbP_{\si^{t,\o}}}[\xi^{t,\o}] = \dbE^{\dbP_0}\Big[\xi^{t,\o}(X^{t,\o})\Big], &\mbox{where}& X^{t,\o}_s = \int_0^s \si^{t,\o}_r(X^{t,\o}_\cd) dB_r,~\dbP_0\mbox{-a.s.}
\eeaa
Since $\si$ is also Lipschitz continuous in $\o$, one can easily show that $u$ is uniformly Lipschitz continuous in $\o$ and, by Proposition \ref{prop:rcpd} with $\l=0$, $u$ is a $\dbP_\si$-martingale.   

We proceed the rest of the proof in three steps.

\ms
\no {\bf 1.}\q We first assume $\si$ is a constant matrix and show that the above $Z$ exists and is bounded. Indeed, by standard approximation again, we may assume $\xi = g(B_{t_1},\cds, B_{t_n})$ for some $0<t_1<\cds<t_n \le T$ and smooth function $g$. Then one can easily see that $u(t, \o) = v(t, B_{t_1},\cds, B_{t_i}, B_t)$, $t_i\le t<t_{i+1}$, for some smooth function $v$. Applying It\^{o}'s formula we obtain the representation with $Z_t = D v(t, B_{t_1},\cds, B_{t_i}, B_t)$, where $D v$ is the gradient in terms of the last variable $B_t$. It is straightforward to check that $D v$ is bounded by the Lipschitz constant of $\xi$, which implies the boundedness of $Z$.

\ms

\no {\bf 2.}\q We now prove the general case.
Denote $\tilde \xi := \xi(X_\cd)$, $\tilde u := u(X_\cd)$, and $\tilde \si := \si(X_\cd)$.  It follows from Lemma \ref{lem:filtration XW} that $\tilde u$ is a $(\dbP_0, \dbF)$-martingale. By the standard martingale representation theorem under $\dbP_0$, there exists $\tilde Z$ such that $\dbE^{\dbP_0}\big[\int_0^T |\tilde Z_t|^2 dt\big]<\infty$ and 
$d\tilde u_t= \tilde Z_t \cd dB_t$, $\dbP_0$-a.s. We claim that
\bea
\label{mrt-claim}
\tilde Z=\tilde \sigma^T \zeta \q \mbox{for some}~ \zeta\in \dbL^0(\dbF, \dbR^d).
\eea
Then
\beaa
d\tilde u_t= \zeta_t \cd dX_t &\mbox{and thus}& d \la \tilde u, X\ra_t = \tilde\si_t \tilde \si^T_t \zeta_t dt, \q\dbP_0\mbox{-a.s.}
\eeaa 
Rewrite $\si = P  \si^* Q$, where $P$, $Q$ are orthogonal matrices and $\si^*= diag[a_1,\cds, a_{d}]$ is a diagonal matrix. Denote $\tilde P := P(X)$, and similarly for other terms. Since $ \la \tilde u, X\ra \in \dbL^0(\dbF^X)$, we see that $ \tilde\si \tilde \si^T \zeta = \tilde P (\tilde\si^*)^2 \tilde P^T \zeta  \in \dbL^0(\dbF^X)$, and thus $(\tilde\si^*)^2 \tilde P^T \zeta \in \dbL^0(\dbF^X)$. Denote $\tilde P^T \zeta := [\zeta_1,\cds, \zeta_d]^T$, and let $\zeta'$ be determined by $\tilde P^T \zeta' := [\zeta_1\1_{\{\tilde a_1 \neq 0\}},\cds, \zeta_d\1_{\{\tilde a_d\neq 0\}}]^T$. Then one can easily check that

$\bullet$ $\tilde \si^* \tilde P^T \zeta'  =  [\tilde a_1\zeta_1\1_{\{\tilde a_1 \neq 0\}},\cds, \tilde a_d \zeta_d\1_{\{\tilde a_d\neq 0\}}]^T= \tilde \si^* \tilde P^T \zeta$ and thus $\tilde \si \zeta' = \tilde \si \zeta$;

$\bullet$ $\tilde a_i^2 \zeta_i \1_{\{\tilde a_i \neq 0\}} \in \dbL^0(\dbF^X)$, then  $\zeta_i \1_{\{\tilde a_i \neq 0\}} \in \dbL^0(\dbF^X)$, thus $\tilde P^T \zeta' \in \dbL^0(\dbF^X)$ and hence  $\zeta' \in \dbL^0(\dbF^X)$.

\no The second property above implies that   $\zeta' = Z(X)$ for some $Z\in \dbL^0(\dbF)$. Then it follows from the first property that
\beaa
d\tilde u_t=\zeta'_t \cd dX_t,\q\dbP_0\mbox{-a.s.}
&\mbox{and thus}& 
d u_t =Z_t \cd dB_t,\q\dbP_\si\mbox{-a.s.}
\eeaa
which is the desired representation.

\ms
\no {\bf 3.} It remains to  prove the claim \reff{mrt-claim}. Consider the decomposition $\tilde Z = \tilde \si^T \zeta + \eta$, where $\tilde\si\eta = 0$, and let us prove that $\eta = 0$, $\dbP_0$-a.s. For this purpose, let $n>0$, $h:= {T\over n}$, $t_i:=i h$, $i=0,\cds, n$, and denote $\bar\eta_i:=h^{-1}\dbE^{\dbP_0}\big[\int_{t_i}^{t_{i+1}}\eta_sds|\cF_{t_i}\big]$,   $\bar\si_i:=h^{-1}\dbE^{\dbP_0}\big[\int_{t_i}^{t_{i+1}}\tilde\si_sds|\cF_{t_i}\big]$, $i=0,\ldots,n-1$. Then,
 \beaa
 \dbE^{\dbP_0}\Big[\int_0^T |\eta_t|^2dt \Big]
 &=&
 \dbE^{\dbP_0}\Big[\int_0^T \tilde Z_t\cd \eta_t dt \Big]
 \;=\;
 \sum_{i=0}^{n-1} \dbE^{\dbP_0}\Big[\int_{t_i}^{t_{i+1}} \tilde Z_t\cd \bar\eta_i dt \Big]
 +R^n_1,
 \eeaa
where $R^n_1\longrightarrow 0$ as $n\to\infty$. Denoting $B^t_s := B_s - B_t$, it follows from the It\^o isometry that
 \beaa
 \dbE^{\dbP_0}\Big[\int_0^T |\eta_t|^2dt \Big]
 &=&
\sum_{i=0}^{n-1} \dbE^{\dbP_0}\Big[\bar\eta_i \cd B^{t_i}_{t_{i+1}} \int_{t_i}^{t_{i+1}} \tilde Z_t \cd dB_t \Big]
 +R^n_1
 \\
 &=&
\sum_{i=0}^{n-1} \dbE^{\dbP_0}\Big[\big(\tilde u_{t_{i+1}}-\tilde u_{t_i}\big)
                                                          \bar\eta_i \cd B^{t_i}_{t_{i+1}}
                                                   \Big]
 +R^n_1
 \\
 &=&
 \sum_{i=0}^{n-1} \dbE^{\dbP_0}\Big[\big(u_{t_{i+1}}(X)-u_{t_i}(X)\big)
                                                           \bar\eta_i \cd B^{t_i}_{t_{i+1}} \Big]
 +R^n_1
 \\
 &=&
 \sum_{i=0}^{n-1} \dbE^{\dbP_0}\Big[\big(u_{t_{i+1}}(X\otimes_{t_i}\bar\sigma_i B^{t_i})
                                                                   -\dbE^{\dbP_0}[u_{t_{i+1}}(X\otimes_{t_i}\bar\sigma_i B^{t_i})|\cF_{t_i}]\big)                                                           \bar\eta_i \cd B^{t_i}_{t_{i+1}} \Big]
 +R^n_2,
 \eeaa
where we used the fact that $B^{t_i}_{t_{i+1}}$ and $\cF_{t_i}$ are $\dbP_0$-independent.   By the uniform Lipschitz continuity of $u$, we see that $R^n_2\longrightarrow 0$ as $n\to\infty$. We further decompose $\bar\eta_i=\bar\sigma_i^T\eps_i+\hat\eta_i$, where $\bar\sigma_i\hat\eta=0$.  Note that, conditionally on $\cF_{t_i}$,   $\bar\sigma_i B^{t_i}$ and  $\hat\eta_i \cd B^{t_i}_{t_{i+1}} $ are $\dbP_0$-independetnt. Then
 \beaa
& \dbE^{\dbP_0}\Big[\int_0^T |\eta_t|^2dt \Big]
 =
R^n_2+\sum_{i=1}^n r^{n}_i,&\\
 &\mbox{where}~
 r^{n}_i
 :=
 \dbE^{\dbP_0}\Big[\big(u_{t_{i+1}}(X\otimes_{t_i}\bar\sigma_i B^{t_i})
                                                                   - \dbE^{\dbP_0}[u_{t_{i+1}}(X\otimes_{t_i}\bar\sigma_i B^{t_i})|\cF_{t_i}]\big)                                                       \bar\sigma_i^T\eps_i \cd B^{t_i}_{t_{i+1}} \Big].&
 \eeaa
 
We now  analyze $r^n_i$. By Step 1, there exists $\g$ bounded by the Lipschitz constant of $u_{t_{i+1}}$ (in terms of $\o$) such that
\beaa
u_{t_{i+1}}(X\otimes_{t_i}\bar\sigma_i B^{t_i})  - \dbE^{\dbP_0}[u_{t_{i+1}}(X\otimes_{t_i}\bar\sigma_i B^{t_i})|\cF_{t_i}] = \int_{t_i}^{t_{i+1}} \g_t \cd \bar\si_i dB^{t_i}_t.
\eeaa
Then
\beaa
|r^n_i| = \Big|\dbE^{\dbP_0}\Big[\int_{t_i}^{t_{i+1}}\g_t dt \cd \bar\si_i\bar\sigma_i^T\eps_i\Big] \Big|= \Big|\dbE^{\dbP_0}\Big[\int_{t_i}^{t_{i+1}}\g_t dt \cd \bar\si_i\bar\eta_i\Big] \Big| \le  C h\dbE^{\dbP_0}\big[|\bar\si_i\bar\eta_i|\big].
\eeaa
 Since $\tilde\si \eta =0$, then
 \beaa
 0 =  \dbE^{\dbP_0}\Big[\int_{t_i}^{t_{i+1}} \tilde \si_t \eta_t dt\Big|\cF_{t_i}\Big]  = \bar\sigma_i \bar \eta_i +  \dbE^{\dbP_0}\Big[\int_{t_i}^{t_{i+1}}\big( [\tilde \si_t - \bar \si_i] \eta_t  + \bar\si_i [\eta_t - \bar\eta_i]\big)dt\Big|\cF_{t_i}\Big].
 \eeaa
 Thus, noting that  $\si \in  C^0_2(\Th)\subset \dbS^2$ and 
$\dbE^{\dbP_0}\big[\int_0^T |\eta_t|^2 dt\big]<\infty$, 
 \beaa
 \sum_{i=1}^n |r^n_i &\le& C\sum_{i=1}^n \dbE^{\dbP_0}\Big[\int_{t_i}^{t_{i+1}}\big| [\tilde \si_t - \bar \si_i] \eta_t  + \bar\si_i [\eta_t - \bar\eta_i]\big|dt\Big]\\
 &\le& C \Big(\sum_{i=1}^n \dbE^{\dbP_0}\Big[\int_{t_i}^{t_{i+1}}[|\tilde \si_t - \bar \si_i|^2 + |\eta_t - \bar\eta_i|^2] dt\Big)^{1\over 2} \to 0,
  \eeaa
as $n\to \infty$. This implies that $ \dbE^{\dbP_0}\Big[\int_0^T |\eta_t|^2dt \Big] = 0$ and thus proves \reff{mrt-claim}.
\qed

\subsection{Some measurability issues}
As a preparation for the nonlinear optimal stopping problem which will be studied in Section \ref{subsec:optimalstop}, we investigate a subtle but crucial measurability issue here. Recall that $\dbF$ is the natural filtration generated by $B$. Denote:
 \bea
\label{dbF*}
\dbF^*:= \mbox{$\dbP_\si$-augmentation of $\dbF$} &\mbox{and}& \cT^* := \mbox{the set of $\dbF^*$-stopping times}.
\eea

We start with  the Blumenthal $0$-$1$ law under $\dbP_\si$. 
\begin{prop}[Blumenthal's 0-1 law]
\label{prop:01}
Under Assumption \ref{assum:si},   for any bounded $\xi\in \cF_{t+}$, $\dbE^{\dbP_{\si}}[\xi|\cF_t] = \xi$, $\dbP_\si$-a.s.
Consequently, the augmented filtration $\dbF^*$  is right continuous. 
\end{prop}
\proof Denote again that $X:=X^\si$ and $\tilde \xi := \xi(X)$. Clearly $\tilde \xi \in \cF_{t+}$, and by the Blumenthal $0$-$1$ law under $\dbP_0$, we have $\dbE^{\dbP_0}[\tilde \xi|\cF_t] = \tilde \xi$, $\dbP_0$-a.s. Since $\tilde\xi \in \dbL^1(\cF^{X}_T, \dbP_0)$, Applying Lemma \ref{lem:filtration XW} we see that $\dbE^{\dbP_0}[\tilde \xi|\cF^X_t] = \tilde \xi$, $\dbP_0$-a.s. which exactly means $\dbE^{\dbP_{\si}}[\xi|\cF_t] = \xi$, $\dbP$-a.s.
\qed

Follow the arguments in \cite{DelacherieMeyer}, we have
\begin{prop}\label{previsible version}
Let  $\t \in \cT^*$ be previsible, namely there exist $\t_n \in \cT^*$ such that $\t_n < \t$ and $\t_n \uparrow \t$. Then there exists $\bar \t \in \cT$ such that $\bar \t = \t$, $\dbP_\si$-a.s.  
\end{prop}
\proof Denote by $\dbF^+:=\{\cF^+_t\}_{0\le t\le T}$ the right filtration. For each $n\ge 1$ and  $r\in \dbQ\cap [0, T]$, denote $E^n_r := \{\t_n < r\} \in \cF^*_r$. Then there exists $\tilde E^n_r \in \cF_r$ such that $\tilde E^n_r \subset E^n_r$ and $\dbP_\si(E^n_r \backslash \tilde E^n_r) = 0$. Note that $E^n_r$ is decreasing in $n$ and increasing in $r$, without loss of generality we may assume that $\tilde E^n_r$ has the same monotonicity. Define
\beaa
\tilde \t_n := \inf\{r \in \dbQ\cap [0, T]: \o \in \tilde E^n_r\} \wedge T,\q \tilde\t := \lim_{n\to \infty} \tilde\t_n.
\eeaa
One can easily check that $\tilde\t_n$ and $\tilde \t$ are $\dbF^+$-stopping times, $\tilde\t_n \uparrow \tilde\t$,  and $\dbP_\si(\tilde \t = \t)=1$. To construct the desired $\dbF$-stopping time, we modify $\tilde\t_n$ and $\tilde\t$ as follows.
\beaa
\bar\t_n := \Big(\tilde \t_n \1_{\{\tilde\t_n < \tilde\t\}} + T \1_{\{\tilde\t_n = \tilde\t\}}\Big) \wedge (T-{1\over n}),\q \bar \t := \lim_{n\to\infty} \bar\t_n.
\eeaa
It is clear that $\bar\t_n$ are also $\dbF^+$-stopping times, $\bar\t_n \uparrow \bar\t$, $\bar\t \ge \tilde\t$, and $\dbP_\si(\bar \t = \t)=1$.  Moreover, for each $n$, on $\{\tilde \t_n<\tilde\t\}$ we have $\bar\t_n = \tilde \t_n \wedge (T-{1\over n}) < \tilde\t \le \bar\t$; and on $\{\tilde \t_n= \tilde\t\}$, we have $\tilde\t_m = \tilde\t$ for all $m\ge n$, thus $\bar \t_m = T-{1\over m}$, $\bar \t = T$, and therefore $\bar\t_n = T-{1\over n} < \bar\t$. So in both cases we have $\bar\t_n < \bar \t$. Then
\beaa
\{\bar\t\le t\} = \cap_{n\ge 1} \{\bar\t_n < t\} \in \cF_t,
&\mbox{for all}&
t\le T.
\eeaa
That is, $\bar\t$ is an $\dbF$-stopping time.
\qed

\begin{lem}
\label{lem:hit}
 Assume $X\in \dbL^0(\dbF)$ is continuous (in $t$), $\dbP_\si$-a.s. Then there exists $\t \in \cT$ such that $\t = \inf\{t: X_t = 0\} \wedge T$, $\dbP_\si$-a.s.
\end{lem}
\proof If $X_0=0$, then  $\t:=0$ satisfies all the requirement. We thus assume $X_0\neq 0$. Set $E:= \{\o: X(\o) ~\mbox{is continuous on}~[0, T]\}$ and $\hat X := X \1_{E} + \1_{E^c}$. Then $\hat X \in \dbL^0(\dbF^*)$ is continuous for all $\o$ and $\hat X_0 \neq 0$. Denote $\hat \t := \inf\{t: \hat X_t = 0\} \wedge T\in \cT^*$ and $\hat\t_n := \inf\{t: |\hat X_t|\le {1\over n}\} \wedge (T-{1\over n})\in \cT^*$. Clearly $\hat\t_n<\hat\t$ and $\hat\t_n \uparrow \hat\t$. By Proposition \ref{previsible version}, there exists $\t\in \cT$ such that $\hat\t = \t$, $\dbP_\si$-a.s. Note that $\t = \inf\{t: X_t = 0\} \wedge T$ on $\{\hat \t = \t\} \cap E$. Since $\dbP_\si[\hat\t=\t]=\dbP_\si[E] = 1$, this concludes the proof.
\qed

\subsection{Optimal stopping under $\ol\cE_L$}\label{subsec:optimalstop}

The next result is a BSDE characterization of the nonlinear expectation $\ol\cE_L$, which extends the $g$-expectation of Peng \cite{Peng-g} to general $\si$.
\begin{prop}
\label{prop:g}
Let $\xi \in \dbL^2(\cF_T, \dbP_\si)$ and  $\t\in \cT$. 

\no {\rm (i)}\q For any $\l\in\dbL^0(\dbF)$ bounded,   $\dbE^{\dbP^{\t,\o}_{\si,\l}}[\xi^{\t,\o}] = Y^\l_\t(\o)$ for $\dbP_\si$-a.e. $\o$, where
\beaa
Y^\l_t = \xi  + \int_t^T Z_s \cd  \si_s \l_s ds - \int_t^T Z_s \cd dB_s,\q \dbP_\si\mbox{-a.s.}
\eeaa

\no {\rm (ii)}\q For any $L>0$, $\ol\cE^{\t,\o}_L[\xi^{\t,\o}] = Y_\t(\o)$ for $\dbP_\si$-a.e. $\o$,  where
\beaa
Y_t = \xi  + \int_t^T L |\si^T_s Z_s|ds - \int_t^T Z_s \cd dB_s,\q \dbP_\si\mbox{-a.s.}
\eeaa
\end{prop}
\proof (i).\q The result follows directly from the definition of $\dbP_{\si,\l}$ and Proposition \ref{prop:rcpd}.

\no (ii).\q Following Proposition \ref{prop:rcpd}, for $\dbP_\si$-a.e. $\o$, we have $Y^{\t,\o}_t = \tilde Y_t$, $0\le t\le \tilde T := T-\t(\o)$, $\dbP_{\si^{\t,\o}}$-a.s. where $\tilde Y$ is the solution to the following shifted BSDE:
\beaa
\tilde Y_t = \xi^{\t,\o} + \int_t^{\tilde T} L|(\si^{\t,\o})^T_s \tilde Z_s|ds - \int_t^{\tilde T} \tilde Z_s \cd dB_s,\q
0\le t\le \tilde T, \q \dbP_{\si^{\t,\o}}\mbox{-a.s.}
\eeaa
Clearly, we have $\tilde Y_0 = \ol\cE^{\t,\o}_L[\xi^{\t,\o}]$, and therefore, $Y_\t(\o)=\ol\cE^{\t,\o}_L[\xi^{\t,\o}]$, $\dbP_\si$-a.s.
\qed

As an application of Proposition \ref{prop:g}, we study the optimal stopping problem under $\ol \cE_L$ via reflected BSDE under $\dbP_\si$:
\bea
\label{RBSDE}
\left\{\ba{lll}
\dis Y_t = X_\ch + \int_t^\ch L|\si^T_s Z_s|ds - \int_t^\ch Z_s \cd dB_s + K_\ch - K_t;\\
\dis Y \ge X,\q  (Y_t-X_t)dK_t =0;
\ea\right.
0\le t\le \ch,\q \dbP_\si\mbox{-a.s.}
\eea
Here the component $K$ of the solution triplet $(Y,Z,K)$ is by definition nondecreasing with $K_0=0$. Given the martingale representation Theorem \ref{thm:mrt}, it follows from standard arguments (see e.g. \cite{EKPPQ}) that \reff{RBSDE} has a unique  solution $(Y,Z,K)\in \dbS^2 \times \dbH^2\times \dbI^2$, restricted on $[0, \ch]$.

We are now ready to establish the nonlinear Snell envelope theory. 

\vspace{3mm}
\no {\bf Proof of Theorem \ref{thm:optimalstop_new}} \q (i)\q Since $X$ and $Y$ are continuous, $\dbP_\si$-a.s., applying Lemma \ref{lem:hit} we have $\t^*\in \cT$ such that $\t^* = \inf\{t: Y_t=X_t\}\wedge \ch$, $\dbP_\si$-a.s.   Moreover, since $Y_{\ch} = X_{\ch}$, it is clear that $Y_{\t^*}= X_{\t^*}$, $\dbP_\si$-a.s. To see the optimality of $\t^*$, we first note that $Y>X$ in $[0, \t^*)$. Then it follows from the minimum condition in \reff{RBSDE} that $K = 0$ in $[0, \t^*)$. Thus RBSDE \reff{RBSDE} becomes a standard BSDE on $[0, \t^*]$. Now it follows from Proposition \ref{prop:g} (ii) that $Y_0 = \ol\cE_L[Y_{\t^*}] = \ol\cE_L[X_{\t^*}]$. 

\ms
\no (ii)\q  We first show that $V_0=Y_0$. For any $\t\in \cT_\ch$, by Proposition \ref{prop:g} (ii) $\ol\cE_L[X_\t] = Y^\t_0$, where 
\beaa
Y^\t_t = X_{\t} + \int_t^{\t} L|\si^T_s Z^\t_s|ds - \int_t^{\t} Z^\t_s \cd dB_s,\q 0\le t\le {\t},~\dbP_\si\mbox{-a.s.} 
\eeaa
Note that $X_{\t}\le Y_{\t}$,  it follows from the comparison principle of BSDEs that $Y^\t_0 \le Y_0$. Then $V_0 \le Y_0$. On the other hand, by (i) we have $Y_0 = \ol\cE_L[X_{\t^*}] \le V_0$. So $Y_0=V_0$. 

For the general case, following Proposition \ref{prop:rcpd}, for any $\t\in\cT_\ch$ and $\dbP_\si$-a.e. $\o$, we have $Y^{\t,\o}_t = \tilde Y_t$, $0\le t\le \tilde \ch := \ch^{\t,\o}-\t(\o)$, $\dbP_{\si^{\t,\o}}$-a.s. where $\tilde Y$ is the solution to the following shifted RBSDE:
\beaa
\left\{\ba{lll}
\dis \tilde Y_t = X^{\t,\o}_{\tilde \ch} + \int_t^{\tilde \ch} L|(\si^{\t,\o})^T_s \tilde Z_s|ds - \int_t^{\tilde \ch} \tilde Z_s \cd dB_s + \tilde K_{\tilde \ch} -\tilde K_t;\\
\dis \tilde Y \ge X^{\t,\o},\q  
(\tilde Y_t-X^{\t,\o}_t)d\tilde K_t =0;
\ea\right.
0\le t\le \tilde\ch, \q \dbP_{\si^{\t,\o}}\mbox{-a.s.}
\eeaa
Then the above arguments (for $t=0$) imply that $V_\t(\o) = \tilde Y_0$, and therefore, $V_\t = Y_\t$,  $\dbP_\si$-a.s. 

\ms
\no (iii)\q We take $\dbP^*:=\dbP_{\si,\l^*}$, where $\l^*$ is so that $(\l^*)^T\si^T Z=L|\si^T Z|$ holds. Then the desired result follows.
\qed

We remark that the optimal stopping problem here relies on the convergence Proposition \ref{prop:DCT} implicitly,  more precisely, the wellposedness of RBSDE \reff{RBSDE} relies on the dominated convergence theorem under $\dbP_\si$. In  \cite{ETZ0} the class $\cP_L$ is non-dominated and we do not have this type of convergence theorem. Consequently, the optimal stopping problem in \cite{ETZ0} is technically much more involved than here.
We also remark that a more direct proof, without involving RBSDEs, can be found in \cite{Survey}.

Also as an application of RBSDE, we may prove Proposition \ref{prop:cEsub}.

\vspace{3mm}
\no {\bf Proof of Proposition \ref{prop:cEsub}}\q (i).\q For any $\t\in \cT$ such that $\t\ge t$. Consider the BSDE:
\beaa
Y_s = u_\t + \int_s^\t L|\si^T_r Z_r|dr - \int_s^\t Z_r\cd dB_r,\q 0\le s\le \t,\q \dbP_\si\mbox{-a.s.}
\eeaa
One may easily show that $Y_t = \ol\cE_L\big[u_\t \big| \cF_t\big]$, $\dbP_\si$-a.s. By (ii) of Proposition \ref{prop:g}, we have $Y_t(\o) = \ol\cE_L^{t,\o}\big[u^{t,\o}_{\t^{t,\o}}\big]$ for $\dbP_\si$-a.e. $\o$. Since $u$ is a pathwise $\ol\cE_L$-submartingale and $\t^{t,\o}\in\cT_{T-t}$, we obtain that
\beaa
u_t(\o)~\le~ \ol\cE_L^{t,\o}\big[u^{t,\o}_{\t^{t,\o}}\big]
 ~= ~ \ol\cE_L\big[u_\t \big| \cF_t\big](\o),\q\dbP_\si\mbox{-a.s.}
\eeaa
Therefore, $u$ is an $\ol\cE_L$-submartingale.

\ms
\no (ii).\q Consider the following RBSDE with upper barrier:
\beaa
\left\{\ba{lll}
\dis Y_t = u_T + \int_t^T L|\si^T_s Z_s|ds - \int_t^T Z_s \cd dB_s - K_T + K_t;\\
\dis Y_t \le u_t,\q 
(u_t-Y_t) dK_t=0;
\ea\right. 0\le t\le T, \dbP_\si\mbox{-a.s.}
\eeaa
Similar to Theorem \ref{thm:optimalstop_new}, one can show that  $Y_t= \einf_{\t\in\cT, \t\ge t} \ol\cE_L[u_\t|\cF_t]$, $\dbP_\si$-a.s. Since $u$ is an $\ol\cE_L$-submartingale, we get $\ol\cE_L[u_\t|\cF_t] \ge u_t$, $\dbP_\si$-a.s. for all $\t\in \cT_{T-t}$, and thus $Y \ge u$. On the other hand, by definition $Y\le u$. Hence, $u=Y$. Further, take $\dbP^*:=\dbP_{\si,\l^*}$, where $\l^*$ is so that $(\l^*)^T\si^T Z=L|\si^T Z|$ holds. Then the desired result follows.
\qed

\end{document}